\documentclass{amsart}
\usepackage{amsmath, amsfonts} % AMS Math Package
\usepackage{amsthm} % Theorem Formatting
\usepackage{amssymb}    % Math symbols such as \mathbb
\usepackage{graphicx} % Allows for eps images
\usepackage[top=1in, bottom=1in, left=1in, right=1in]{geometry}
\usepackage{mathtools}
\usepackage{mathrsfs}
\usepackage[colorlinks=true,linkcolor=black,anchorcolor=black,citecolor=black,filecolor=black,menucolor=black,runcolor=black,urlcolor=black]{hyperref}
\usepackage{enumerate}
\usepackage{longtable}
\usepackage{color,comment}
\usepackage{extarrows}
\usepackage{centernot}
\usepackage{tikz}
\usetikzlibrary{trees,arrows,positioning}
\usepackage{tikz-cd}

\theoremstyle{plain}
\newtheorem{prop}{Proposition}[section]
\newtheorem{thm}[prop]{Theorem}
\newtheorem{main}{Theorem}[section]
\newtheorem{lem}[prop]{Lemma}
\newtheorem{lemma}[prop]{Lemma}
\newtheorem{cor}[prop]{Corollary}
\newtheorem{rem}[prop]{Remark}
\theoremstyle{definition}
\newtheorem{dfn}[prop]{Definition}
\newtheorem{example}[prop]{Example}

\numberwithin{equation}{section}
\numberwithin{figure}{section}

\hypersetup{colorlinks=true,linkcolor=magenta,citecolor=purple}

\newcommand{\mycomment}[1]{}
\newcommand{\Gl}{\mathrm{GL}}
\newcommand{\Ad}{\mathrm{Ad}}

\DeclareMathOperator{\Sl}{SL}
\newcommand{\g}{\mathfrak{g}}
\newcommand{\cO}{\mathcal{O}}

\DeclareMathOperator{\Ind}{Ind}

\DeclareMathOperator{\diag}{diag}

\DeclareMathOperator{\cInd}{c-Ind}
\DeclareMathOperator{\Indd}{Ind}
\DeclareMathOperator{\Res}{Res}
\DeclareMathOperator{\Hom}{Hom}
\DeclareMathOperator{\Tr}{Tr}

\DeclareMathOperator{\val}{val}

\DeclareMathOperator{\K}{\mathcal{K}}

\DeclareMathOperator{\End}{End}

\DeclareMathOperator{\RR}{\mathcal{R}}
\DeclareMathOperator{\PP}{\mathcal{P}}
\DeclareMathOperator{\vol}{vol}
\newcommand{\cS}{\mathcal{S}}

\newcommand{\ep}{\varepsilon}
\newcommand{\triv}{\mathbf{1}}
\newcommand{\WF}{\mathrm{WF}}
\newcommand{\red}{\mathrm{red}}

\newcommand{\Z}{\mathbb{Z}}
\newcommand{\buil}{\mathscr{B}}

\newcommand{\Fqt}{\mathbb{F}_q(\!(t)\!)}
\newcommand{\Fpt}{\mathbb{F}_2(\!(t)\!)}

\newcommand{\Fqtint}{\mathbb{F}_q[\![t]\!]}

\newcommand{\maxKSL}{{\mathcal{K}'}}
\newcommand{\maxotherKSL}{{\prescript{g_1}{}{\mathcal{K}'}}}

\newcommand{\gone}{g_1}  

\newcommand{\gell}{g_\ell}

\newcommand{\sigmaell}{\prescript{g_\ell}{}{\sigma}}

\newcommand{\Jtype}{\K_{m,m'}}
\DeclareMathOperator{\intgroup}{\Gamma(\ell)}

\newcommand{\dcst}{\mathscr{S}}
\DeclareMathOperator{\car}{char}
\newcommand{\intersec}{\prescript{\gamma}{}{B'_\ell}\cap B'_\ell}  %hard to type
\DeclareSymbolFont{bbold}{U}{bbold}{m}{n}
\DeclareSymbolFontAlphabet{\mathbbold}{bbold}

\begin{document}

\title[Branching Rules in residual characteristic $2$]{Branching rules for irreducible depth-zero supercuspidal representations of $\Sl(2,F),$ when $F$ has residual characteristic $2.$}
\author{Zander Karaganis}
\address{Department of Mathematics and Statistics, University of Ottawa, Ottawa, Canada K1N 6N5}
\email{zander.karaganis@mail.utoronto.ca}

\date{\today}
\author{Monica Nevins}
\address{Department of Mathematics and Statistics, University of Ottawa, Ottawa, Canada K1N 6N5}
\email{mnevins@uottawa.ca}
\thanks{The second author's research is supported NSERC Discovery Grant RGPIN-2025-05630.}


\begin{abstract}
    We give the decomposition into irreducible representations of the restriction to a maximal compact subgroup of any irreducible depth-zero supercuspidal representation of $\Sl(2,F)$ when $F$ is a local nonarchimedean field of residual characteristic two.  We furthermore provide explicit constructions of these irreducible components in terms of nilpotent orbits, proving a representation-theoretic analogue of the local character expansion that holds even in the wild case of characteristic two.
\end{abstract}

\maketitle

\section{Introduction}

A $p$-adic group $G$ is the group of $F$-points of a reductive algebraic group defined over a local nonarchimedean field $F$ of residual characteristic $p$.  The restriction of a smooth irreducible complex representation $\pi$ of a $p$-adic group $G$ to a maximal compact open subgroup $K$ provides a rich array of data about $\pi$ --- from its Bushnell--Kutzko types \cite{Latham2017,LathamNevins2021}, to its Gelfand--Kirillov dimension \cite{BarbaschMoy1997}, or its asymptotic behaviour near the identity \cite{Nevins2024,HenniartVigneras2024}.  The representation theory of $K$ remains an open problem and consequently, the complete set of these branching rules has only been obtained in a handful of cases of rank one groups including particularly: $\Gl(2,F)$ \cite{Cassel,Han}; $\Sl(2,F)$ \cite{Nev} and its $n$-fold covering groups \cite{Karimianpour2018} assuming $p\neq 2$; and unramified $U(1,1)$ \cite{Tiwari2025}, again with $p\neq 2$.

The case $p=2$ has often been unattainable due to the arithmetic complexity arising from the wildness of quadratic extensions of $F$.  In this paper, we 
determine the complete branching rules of all depth-zero supercuspidal representations of $G'=\Sl(2,F)$ over a local nonarchimedean field $F$ of residual characteristic $p=2$.  Unlike the case when $p$ is odd, the number of irreducible components of every depth is not constant, and in fact it grows without bound when $\car(F)=2$.  Nevertheless, we prove that these components admit an elegant description in terms of the geometry of the nilpotent elements of the Lie algebra $\mathfrak{g}'$.  Using these, we derive two kinds of representation-theoretic versions of the local character expansion, that is, simple expressions of the restriction of $\pi$ to a sufficiently small neighbourhood of the identity as a linear combination of representations arising from nilpotent orbits.

Along the way, we carefully develop a number of tools and techniques that generalize far beyond the current setting, and we expect that, as in the case of $p$ odd, the representations constructed here will exhaust the branching rules of a general irreducible representation of $G'$, up to a finite-dimensional piece.

Our results fit into the theory of the local nature of representations presented by Henniart and Vign\'eras in \cite{HenniartVigneras2024,HenniartVigneras2025} and provide an explicit sharp bound on the neighbourhood on which their local expansion holds (for  representations over $\mathbb{C}$).   Our theorems, stated for all primes $p$, specialize to the  main results of \cite{Nevins2024} when $p\neq 2$ for depth-zero supercuspidal representations, fully incorporating the arithmetic surprises that have thus far kept $p=2$ from full exploration.

We state our  main theorems as follows.  Let $G'=\Sl(2,F)$ and $\K'=\Sl(2,\RR)$ be a maximal compact open subgroup.  Since   the second conjugacy class of maximal compact open subgroups of $G'$ is represented by $\prescript{g_1}{}{\K'}$, where $g_1 = \diag(\varpi,1) \in G=\Gl(2,F)$, it suffices to establish the branching rules for restriction to $\K'$.

\begin{main}[Theorem~\ref{T:brulesSL2}]\label{T:1}
    Suppose $\pi$ is a depth-zero supercuspidal representation of $G'$. If it is of the form $\pi_0(\sigma)=\cInd_{\K'}^{G'}\sigma$ for some  cuspidal representation $\sigma$ of $\K'/\K'_+$, then 
    $$
    \Res_{\K'}\pi_0(\sigma) = \sigma \oplus \bigoplus_{\ell \in 2\mathbb{Z}_{>0}} \bigoplus_{u\in \RR^\times/(\RR^\times)^2 (1+\PP^{\ell/2})} I(\triv, u,\ell)
    $$
    where $I(\triv, u,\ell)$ is an irreducible representation of depth $\ell$, defined in \eqref{E:defIzetauell}.  Otherwise,  $\pi \cong \pi_1(\sigma') =\cInd_{\prescript{g_1}{}{\K'}}^{G'}\sigma'$ for some cuspidal representation $\sigma'$ of $\prescript{g_1}{}{\K'}/\prescript{g_1}{}{\K'_+}$ and
    $$
    \Res_{\K'}\pi_1(\sigma') = \bigoplus_{\ell \in 1+2\mathbb{Z}_{\geq 0}} \bigoplus_{u\in \RR^\times/(\RR^\times)^2 (1+\PP^{(\ell+1)/2})} I(\triv, u,\ell).
    $$
\end{main}

We prove this by first applying Mackey theory to write $\Res_{\K'}\pi$ as a direct sum of  (reducible) \emph{Mackey components} $\sigma(\ell)$ in Section~\ref{S:4}, and we determine their intertwining in Section~\ref{Sec: Branching Rules}. After a brief interlude in Section~\ref{S:6} to derive some consequences when $q=2$, we construct in Section~\ref{S:7}, for each $\ell>0$,  representations $I(\triv, u,\ell)$ of $\K'$ of depth $\ell$, arising from (the reduction mod $\g'_{x,-\ell/2}$ of) a nilpotent $\K'$-orbit $\cO_u$ of depth $-\ell$ in the Lie algebra of $G'$.  We then prove in Theorem~\ref{T:JzetaellIzetauell} that these representations are irreducible and find their intertwining with $\sigma(\ell)$, yielding the decomposition of $\sigma(\ell)$ into irreducible $\K'$-representations.

Our next goal is to prove  that for depth-zero supercuspidal representations of $\Sl(2,F)$, the analytic character expansion (which exists when $\car(F)\neq 2$) can be expressed as a statement in the Grothendieck group of representations in an explicitly-determined neighbourhood of the identity. We propose two variants of the theorem; taken together with \cite[Theorem 1.1]{Nevins2024}, the first of these gives the following.
\begin{main}[Theorem~\ref{T:LCE}]\label{T:2}
    Let $F$ be a $p$-adic field with $\car(F) = 0$. Then to each nilpotent $\Sl(2,F)$-orbit $\cO$ in $\mathfrak{sl}(2,F)$ we may associate a representation $\tau(\cO)$ of $\K'=\Sl(2,\RR)$, and to each irreducible depth-zero  supercuspidal representation $\pi$ of $\Sl(2,F)$ we may associate a set of nilpotent orbits $\WF(\pi)$, such that
    \begin{equation*}
        \pi|_{\K_{r+}'} = n\cdot \triv + \sum_{\cO \in \WF(\pi)} \tau(\cO)|_{\K'_{r+}}
    \end{equation*}
    where $r=4\val(2)$, that is, $r=0$ when $p$ is odd.
\end{main}

This result also suggests a bound for the domain of validity for the identity \cite[Corollary 6.14]{HenniartVigneras2025} for all $F$, which expresses $\pi$ instead as an integral linear combination of the representations in an $L$-packet of $\Sl(2,F)$ of size four.

Our second locality result is valid also for fields of characteristic two, and generalizes \cite[Theorem 7.4]{Nevins2024} to this setting. A   $(-\ell, -\ell/2)$ degenerate coset is a coset $X+\g'_{x,-\ell/2}$, with $X$ of depth $-\ell$ at $x$,  meeting one or more nilpotent $G'$-orbits; here we suppose $x \in \buil(G)$ is fixed by $\K'$.  With Definition~\ref{D:tauzetauell} we attach to each such coset an infinite-dimensional representation $\tau_{\triv,u,\ell}$.

\begin{main}[Theorem~\ref{T:LCE2}]\label{T:3}
Let $\pi = \pi_i(\sigma)$ be a depth-zero supercuspidal representation of $\Sl(2,F)$ where $\car(F)\in \{0,2\}$, $p=2$ and $i\in \{0,1\}$.  Then for any $\ell>0$ such that $\ell \in i+2\mathbb{Z}$ we have 
$$
\Res_{\K'}\pi \cong \pi^{\K'_{\ell}} \oplus \bigoplus_{u\in \cS_{\lceil \ell/2 \rceil}} \tau_{\triv,u,\ell}.
$$
The number of terms $\tau_{\triv, u, \ell}$ is finite and they index the distinct $(-\ell, -\ell/2)$ degenerate cosets at $x$.  There are $2q^e+1$ summands if $\ell\geq 4e+1$ but this number grows to infinity with $\ell$ when $\car(F)=2$.
\end{main}

 This theorem expresses that, independent of the characteristic of $F$,  $\Res_{\K'}\pi$ is completely determined by the local geometry of the nilpotent cone in \emph{every} neighbourhood of the identity, up to a finite-dimensional subrepresentation  $\pi^{\K'_\ell}$ whose depth controls the resolution of the decomposition.  The summands once again correspond to elements of $\WF(\pi)$, as defined in Definition~\ref{D:WFpi}.  

Along the way to these results we prove far more towards our goal of developing tools for the branching rules of more general representations in residual characteristic two, as well as insight into the key arithmetic obstructions that have made this case appear intractable until now. For one, we also address the (simpler) case of $\Gl(2,\RR)$, whose branching rules were determined by Hansen in \cite{Han}, providing new insights into her results.  For another, we contrast the methods of this paper to the solved case of $p$ is odd (\cite{Nev,Nevins2024}) throughout; when suitably interpreted, we recover the results for  $p$ odd as a special case.

Several interesting questions remain open.  Having constructed the family of irreducible representations $I(\zeta,u,\ell)$ in Section~\ref{S:7}, we anticipate that these should form the bulk of the branching rules for any irreducible representation of $\Sl(2,F)$, as was the case when $p\neq 2$ \cite[Theorem 4.1]{NevinsPatterns}, and consistent with the expectations from the local character expansion.  It is, however, challenging to explicitly detail the representations of $\Sl(2,F)$ (see \cite{Kut2,KutzkoPantoja1991}), let alone to compute their branching rules.  When $p$ is odd, the representation theory of $\K'$ is known (\cite{Shalika}); our representations $I(\zeta,u,\ell)$ are a novel contribution.

In another direction, recent work by Labesse \cite{Labesse2025} uses the endoscopic expansion of elliptic orbital integrals to produce a well-defined analogue of the germ expansion of a semisimple element when $\car(F)=2$.  This is a very promising development, since the work of Kim--Murnaghan \cite{KimMurnaghan2003} reduces the local character expansion for positive-depth supercuspidal representations to the germ expansion of a semisimple element and this was exploited in \cite{Nevins2024}.

Our paper is organized as follows.  We set out notation in Section~\ref{S:2} and  in Section~\ref{S:3} we take a deep dive into local fields of residual characteristic two, focussing particularly on squaring in $F^\times$ and in $\Sl(2,F)$.  In Section~\ref{S:4} we recall the construction of the depth-zero supercuspidal representations of $\Sl(2,F)$ and do the first step of the decomposition of $\Res_{\K'}\pi$ into representation of depth $\ell$, denoted $\sigma(\ell)$, leveraging results of Hansen \cite{Han} for $G=\Gl(2,F)$ (which were uniform across all $p$).  Our key technical result from Section~\ref{S:3}, Proposition~\ref{L:howcloseadare}, is  applied in Section~\ref{Sec: Branching Rules} to prove that the number of self-intertwining operators $\dim\End_{\K'}(\sigma(\ell))$ grows in bijection with the number of square classes modulo $\PP^{\lceil \ell/2\rceil}$ (Theorem~\ref{dimHomeSigmaEll} and Corollary~\ref{C:Sigmaell}).   In Section~\ref{S:6}, we demonstrate how to use Mackey theory  to explicitly prove that each $\End_{\K'}(\sigma(\ell))$ is abelian; in this part only we assume the residue field of $F$ is $\mathbb{F}_2$, for simplicity, and the results of Section~\ref{S:7} are independent of Section~\ref{S:6}.  

Our in-depth treatment of nilpotent orbits of $\mathfrak{sl}(2,F)$ (of which there are infinitely many, when $\car(F)=2$) in Section~\ref{nilp} sets the stage for our main theorems.  We construct irreducible representations $I(\zeta,u,\ell)$ of $\Sl(2,\RR)$ and $J(\zeta,\ell)$ of $\Gl(2,\RR)$ in Section~\ref{SS:Iuell}, and prove Theorem~\ref{T:1} in Section~\ref{SS:brules}.
In Section~\ref{S:8} we derive several applications of our results.  The first, in Section~\ref{SS:8.1}, inspired by the questions posed in \cite{HenniartVigneras2024}, is about the growth rates of $\pi^{\K'_n}$ and of the maximal irreducible subrepresentation of $\pi^{\K'_n}$, as $n\to \infty$ (Proposition~\ref{P:dn}).   In Section~\ref{SS:LCE1} we define $\WF(\pi)$ and the representations $\tau(\cO)$, and prove Theorem~\ref{T:2}; we prove the analogous result for $G=\Gl(2,F)$ in Section~\ref{SS:GL2LCE}.  We set up and prove Theorem~\ref{T:3} in Section~\ref{SS:LCE2}.
Finally, in Section~\ref{SS:8.4Q2} we detail our results for the special case of $F=\mathbb{Q}_2$ --- some numerology to serve as an enticement to  explore further.

\subsection*{Acknowledgements}  The first author would like to thank the Lisgar Collegiate Institute for providing such a stimulating high school co-op experience.  The second author likewise thanks LCI, as well as  the support of the Institut Henri Poincar\'e (UAR 839 CNRS-Sorbonne Universit\'e) (and the grant number ANR-10-LABX-59-01 in the metadata).  The second author's research is supported NSERC Discovery Grant RGPIN-2025-05630.

\section{Notation and background}\label{S:2}

\subsection{The field}
Let $F$ be a local nonarchimedean field of residual characteristic equal to two.

Suppose first that $\car(F)=0$.  Then $F$ is a $2$-adic field, that is, a finite algebraic extension of $F_0=\mathbb{Q}_2$, the field of $2$-adic numbers.  Write $\mathcal{R}$ for the ring of integers of $F$ with maximal ideal $\mathcal{P}$. Denote the residue field of $F$ by $\mathfrak{f}= \mathcal{R}/\mathcal{P}$; it is isomorphic to $\mathbb{F}_q$ where $q=2^f$ for some $f\in \mathbb{N}$. We fix a uniformizer $\varpi$ of $F$ and normalize the valuation so that $\val(\varpi)=1;$ thus $2=\iota \varpi^e$ for some unit $\iota\in\RR^\times.$
Then $e$ coincides with the ramification index of $F$ over $\mathbb{Q}_2$ and $ef=[F:F_0]$. 

If instead $\car(F)=2$, then 
$F=\Fqt$ where $q=2^f$ for some $f\in \mathbb{N}$ and $t$ is an indeterminate.  
The ring of integers is $\mathcal{R}=\Fqtint$ and the maximal ideal is $\PP=t\Fqtint$.  The residue field is $\mathfrak{f}=\mathbb{F}_q$ and we set $\varpi=t$, normalizing the valuation by $\val(t)=1$. 
As $2=0$, we set $\iota=0$ and $e=\infty$, the latter to be understood as the statement ``$m<e$" is true and ``$m\geq e$" is false.  This is distinct from the ramification degree of $F$ over any subfield.

The results in this paper hold for all such $F$.  Our primary technical focus is on the more challenging case of a $2$-adic field,  and we often provide examples in the context of $F_0=\mathbb{Q}_2$.  

\subsection{Groups and representations} If $\mathbb{G}$ is a connected reductive algebraic group defined over $F$, write $G=\mathbb{G}(F)$ to denote the group of $F$-rational points of $\mathbb{G}$. Where this can cause no confusion, we may simply say that, for example, $B$ is a Borel subgroup of $G$, to mean that $B=\mathbb{B}(F)$ where $\mathbb{B}$ is a Borel subgroup of $\mathbb{G}$ defined over $F$. 

Given a group $G$, a subgroup $H$, and an element $g\in G$, we write $\prescript{g}{}{H} = gHg^{-1};$ likewise, if $\rho$ is a representation of $H$, we write $\prescript{g}{}{\rho}$ for the representation of $\prescript{g}{}{H}$ defined by $\prescript{g}{}{\rho}(h) = \rho(g^{-1}hg).$   
 All representations $(\pi,V)$ of $G$ are assumed to be smooth and complex, that is, $V$ is a complex vector space and for all $v\in V$ there exists a compact open subgroup $K\subset G$ fixing $v$.
 
For any closed subgroup $H$ of $G$ and representation $(\sigma, W)$ of $H$, we define the compact induction $\cInd_H^G \sigma$ of $\sigma$ from $H$ to $G$ by right translation of $G$ on the space 
$$
\cInd_H^G W := \left\{f:G\to W\middle\vert \begin{array}{c}
    \forall h\in H, g\in G, f(hg)=\sigma(h)f(g), \text{$f$ is smooth}\\ \text{and compactly supported modulo $H$} 
\end{array}\right\}
$$ 
of locally constant functions with compact support in the quotient $H\backslash G.$
The restriction of such a representation to a compact open subgroup can be described using Mackey theory.  The following statement is from \cite{Kutzko1977}.

\begin{prop}\label{P:Mackey}
    Let $G$ be the $F$-points of a connected reductive algebraic group. Suppose that $H$ and $L$ are subgroups of $G$ such that $H$ is   compact-mod-centre and $L$ is either closed, or compact open. If $\varrho$ is a representation of $H$ such that $\cInd_{H}^G\varrho$ is admissible, then $$\Res_L \cInd_{H}^G \varrho \cong \bigoplus_{g\in L\backslash G/H} \Indd_{{}^g H \cap L}^L {}^g\varrho.$$ 
\end{prop}
We call the summands --- which are not necessarily irreducible --- the \emph{Mackey components} of the restriction.

We also use Clifford theory (for finite groups, since every smooth irreducible representation of a compact open subgroup factors through a finite quotient).        Let $K$ be a compact open subgroup of $G$ and $N$ a normal subgroup of $K$ of finite index.  For any irreducible representation $\lambda$ of $N$ let $N_K(\lambda)=\{k\in K: \prescript{k}{}{\lambda}\cong \lambda\}$. 

\begin{thm}\label{T:Clifford}
    In the setting above, if  $\pi$ is an irreducible smooth representation of $K$ and  $\Hom_N(\lambda,\Res_N\pi)\neq 0$, then there exists an irreducible representation $\sigma$ of $N_K(\lambda)$ such that $\pi \cong \Ind_{N_K(\lambda)}^K\sigma$,  and
the restriction of $\pi$ to $N$ is a direct sum (possibly with multiplicity) of $K$-conjugates of $\lambda$.
    In particular, all irreducible representations occurring in $\Res_N \pi$ are of equal degree.
\end{thm}

\subsection{Specific notation}
From now on, we set $G=\Gl(2,F)$ and  $\K=\Gl(2,\mathcal{R})$. 
Write $Z=Z(G)$ to denote the center of $G$ and $B$ for the \emph{lower} triangular subgroup.  Write $\K_\ell$ for the $\ell$th congruence subgroup of $\K$, for any $\ell> 0$; then $\K_+=\K_1$.  Write $B_\ell := (B\cap \K)\K_\ell$ for the matrices of $\K$ that are lower triangular modulo $\PP^\ell$.  

Our main focus is the subgroup  $G' = \Sl(2,F)$. In general, if $H$ is a subgroup of $G$, we will use $H'$ to denote $H\cap G'.$  Thus $Z'=\{\pm I\}$, $\K'=\Sl(2,\mathcal{R})$,   and $B'_\ell = (B\cap \K')\K'_\ell$.   We write $\mathfrak{g}$ for the Lie algebra of $G$ and $\mathfrak{g}'$ for that of $G'.$ 

Write $\lceil r \rceil = \min\{n\in \mathbb{Z}\mid n\geq r\}$ and $\lfloor r \rfloor = \max\{n\in \mathbb{Z}\mid n\leq r\}$.  Then $\lceil r +\rceil := \min\{n\in \mathbb{Z}\mid n> r\} = \lfloor r\rfloor+1$.  For a real number $r$, define $\mathcal{P}^r := \{ x\in F\mid \val(x) \geq r\}=\PP^{\lceil r \rceil}$. This gives a filtration of $F$ of the form $\dots \supset \mathcal{P}^{-2} \supset \mathcal{P}^{-1} \supset \mathcal{R} \supset \mathcal{P} \supset \mathcal{P}^2 \supset \cdots .$ The group of units $\mathcal{R}^\times$ of $\mathcal{R}$ similarly admits a filtration by subgroups $1+\PP^m$ for $m\in \mathbb{Z}_{>0}$. 
 
Given sets $S_i,$ we may simply write 
$$
\begin{bmatrix}
    S_1 & S_2\\ S_3 & S_4
\end{bmatrix}:= \left\{\begin{bmatrix}
    s_1 & s_2\\ s_3 & s_4
\end{bmatrix} \middle\vert s_i \in S_i \right\}
$$
to represent the corresponding subgroups of $G$ or $G'$ given by intersection. When the $S_i$ are $\RR$-modules this notation can represent the $\RR$-points of a group scheme. 

Let $\diag(a,b)$ denote the diagonal matrix with entries $a,b$ from $F$.  Some other recurring matrix forms are 
$$
w=\begin{bmatrix}
    0& 1\\ -1&0
\end{bmatrix}, \quad 
g_\ell = \begin{bmatrix} \varpi^\ell &0\\ 0&1\end{bmatrix}, \quad
g(k,\alpha)=
\begin{bmatrix}1&\alpha \varpi^k \\ 0&1\end{bmatrix}, \quad \text{and} \quad X_u=\begin{bmatrix}0& 0\\ u&0\end{bmatrix}
$$ 
with $\alpha,u\in F$ and $k,\ell\in \mathbb{Z}$, representing, respectively:  a Weyl element of $G'$; a double coset of $\mathcal{K}\backslash G/Z\mathcal{K}$; a certain coset of $G'/B'$; and a nilpotent element of the Lie algebra.  We fix throughout an additive character $\psi$ of $F$ with conductor $\mathcal{P}$.

 If $\buil(G)$ denotes the Bruhat--Tits building of $G$, then we have $\buil(G')=\buil^{\red}(G)$, its reduced building.   
 Write $G_x$ for the stabilizer in $G$ of $x\in \buil(G)$ and  $G_{x,r}$, for $r\geq 0$, for the Moy--Prasad filtration subgroups of $G$ at $x$, with the convention that $G_{x,r+}:=\bigcup_{s>r}G_{x,s}$.  We define  $\RR$-subalgebra filtrations $\mathfrak{g}_{x,r}$ of $\g$ similarly, indexed this time by $r\in \mathbb{R}$.  For any $x\in \buil(G)$ and any $0<r/2\leq s<r$, the Moy--Prasad isomorphism
\begin{equation}\label{E:MP}
G_{x,s}/G_{x,r} \to \mathfrak{g}_{x,s}/\mathfrak{g}_{x,r}
\end{equation}
is given by the map $k\mapsto k-I$, independently of $x$ (or of  $p$).   This map factors through to the isomorphism $G'_{x,s}/G'_{x,r}\to \mathfrak{g}'_{x,s}/\mathfrak{g}'_{x,r}$, sending $k$ to the unique coset meeting $k-I + \mathfrak{g}_{x,r}$.

 We designate $x_0\in \buil(G')$ to be the vertex for which $G'_{x_0}=G'_{x_0,0}=\K'$ and denote also by $x_0$ any preimage in $\buil(G)$, so that $G_{x_0}=G_{x_0,0}=\K$.  There are two $G'$-conjugacy classes of vertices in $\buil(G')$ but they are $G$-conjugate: setting $g_1=\diag(\varpi,1)\in G$ as above and $x_1=g_1\cdot x_0\in \buil(G')$, we have $G'_{x_1} = \prescript{g_1}{}{\K'} = \left[\begin{smallmatrix} \RR&\PP\\ \PP^{-1}&\RR\end{smallmatrix}\right]$.  This conjugation preserves the level of filtrations.  Note that we simply have $G_{x_0,r}=\K_{\lceil r \rceil}$, the $\lceil r\rceil$th congruence subgroup, for each $r > 0$.

Given a representation $(\pi,V)$ of $G,$ we define $V^{G_{x,d+}} = \{v\in V\mid \pi(k)v=v, \,\, \forall k\in {G}_{x,d+}\}$. Then in \cite{MoyPrasad1996} 
Moy and Prasad defined the depth of $\pi$ as $$d=d(\pi):= \min\{d\in\mathbb{R}_{\geq 0}\mid \text{ there exists } x\in\mathscr{B}(G) \text{ such that } V^{G_{x,d+}} \neq \{0\}\}.$$
Similarly, the depth of a representation $(\sigma,V)$ of $\K$ is the least integer $d\geq 0$ such that $V^{\K_{d+}}\neq \{0\}.$

\section{On squaring in local fields of residual characteristic two}\label{S:3}

\subsection{Square classes}

The group of square classes of $F$ is given by $F^\times/F^{\times 2} \cong \mathcal{R}^\times/(\mathcal{R}^\times)^2\times \Z/2\Z$, where the second factor is the parity of the valuation.  

\begin{lem}\label{Lemma:Cassel}
    Let $\mathscr{F}$ be a set of representatives for $\mathfrak{f}$ in $\mathcal{R}$.  Choose $\alpha \in \mathfrak{f}^\times$ that is not in the image of the map $x\mapsto x^2+x$ and let $\aleph\in \mathscr{F}$ be a lift of $\alpha$.   If $\car(F)=0$ then the order of $\mathcal{R}^\times/(\mathcal{R}^\times)^2$ is $2q^e$ and a set of representatives is  
    $$
    \cS=\{1 +a_1\varpi + a_2\varpi^{3} + \cdots + a_e\varpi^{2e-1} + 4\gamma \mid \gamma \in \{0,\aleph\}\; \text{and}\; \forall i, a_i \in \mathscr{F} \},
    $$
    whereas if $\car(F)=2$ then $\mathcal{R}^\times/(\mathcal{R}^\times)^2$ is infinite and a set of representatives is
    $$
    \cS = \{1 + \sum_{i\in \mathbb{Z}_{\geq 0}}a_i\varpi^{2i+1}  \mid \forall i, a_i \in \mathscr{F}\}.
    $$     
    %In particular, given $a,b\in \RR^\times$, if $a-b\in  4\PP=\PP^{2e+1}$ then $a\in b(\RR^\times)^2$.
\end{lem}

If $\car(F)=0$, then every element of $1+\PP^{2e+1}=1+4\PP$ is a square and the above lemma is proven in \cite{Cas2023}. If $\car(F)=2$, then we have directly that
 $\RR^{\times 2} = \{\sum_{j\geq 0} a_j\varpi^{2j}\mid a_j\in \mathscr{F}, a_0\neq 0\}$, whence the result.  
 
\begin{example}
    Suppose $F=F_0$.  If we choose $\aleph=1$, $\varpi=2$, $\epsilon=1+\varpi^2=5$ and $\eta=-1\in 1+\varpi+\varpi^2+\mathcal{P}^3$, then the $4q^e=8$ square classes of $F_0^\times$ are
\begin{equation}\label{squareclassesQ2}
F_0^\times/(F_0^\times)^2 = \{1, \epsilon, \eta,  \epsilon\eta, \varpi, \epsilon\varpi, \eta\varpi\}=\{\pm 1, \pm 2, \pm 5, \pm 10\}.
\end{equation}
As $\car(F_0)=0$, the nontrivial classes parametrize the  seven distinct quadratic extensions of $F_0$, with $E=F_0[\alpha]$ with $\alpha \in (\cS\smallsetminus \{1\}) \cup \varpi\cS$.  Note that $\alpha=\epsilon=5\equiv -3 \mod 8$ generates the unramified extension (which is characterized as containing a primitive cube root of unity) but the remaining extensions are (wildly) ramified.
\end{example}

In contrast, when $p$ is odd, $|\RR^\times/(\RR^\times)^2|=2$ and its representatives are distinct mod $\PP$.

\subsection{On products and squaring}
We will require certain matrix calculations over $2$-adic fields to prove the main theorem of Section~\ref{Sec: Branching Rules}.
We begin with a simple result.

\begin{lemma}\label{L:squaring}
    Let $\delta\in \mathbb{Z}_{>0}$ and suppose $a\in \RR^\times$ satisfies $a^2\in 1+\PP^\delta$.  Then $a\in \pm 1 + \PP^{\max\{\delta-e,\lceil \delta/2\rceil\}}$.
\end{lemma}
\begin{proof}
Suppose $a\neq 1$.  Since $a^2\in 1+\PP$ we may write $a=1+z\varpi^k$ for some $k\geq 1$
 and $z\in \RR^\times$, so that  
\begin{equation}\label{E:squaringwithk}
a^2= 1+2z\varpi^k+z^2\varpi^{2k}=1+\iota z\varpi^{k+e} + z^2\varpi^{2k}.
\end{equation}
If $\car(F)=2$ then $\iota=0$ so this lies in $1+\PP^\delta$ if and only if $k\geq \delta/2$; since $e=\infty$ and $1=-1$, this yields the statement for this case.  Suppose now $\car(F)=0$.  If $k\neq e$ then from \eqref{E:squaringwithk} we conclude that $k+e\geq \delta$ and $2k\geq \delta$, as required.  If $k=e$, then \eqref{E:squaringwithk} simplifies to
$$
a^2= 1+ (z\iota + z^2)\varpi^{2e}.
$$
If $2e\geq \delta$ there is nothing to show.  If $2e<\delta$ then we must have $z\equiv \iota\mod \PP$, whence $a\equiv -1 \mod \PP^{e+1}$, or $-a\in 1+\PP^{e+1}$. Since its square lies in $1+\PP^\delta$, we infer from the preceding that $-a\in 1+\PP^{\max\{\delta-e,\delta/2\}}$.
\end{proof}

What we require in further calculations is more subtle.  Let us present the easier case of $\car(F)=2$ first.

\begin{lem}\label{L:howcloseadareCHAR2}
Suppose $F$ is of characteristic $2$. Let $\delta \in \mathbb{Z}_{>0}$.  
Suppose $a,d \in 1+\PP$ satisfy $a\equiv d \mod \PP^\delta$ and $ad \equiv 1 \mod \PP^{\delta+1}$.  Then we may write
\begin{equation}\label{E:char2adexpression}
a = 1 + \sum_{i\geq\delta/2}a_i\varpi^i, \quad \text{and}\quad
d = a + \sum_{i\geq \delta} (d_i-a_i) \varpi^i
\end{equation}
for some $a_i, d_i\in \RR$ and moreover
\begin{itemize}
    \item if $\delta$ is even, then  $d_{\delta} - a_\delta\in a_{\delta/2}^2 + \PP$;
    \item if $\delta$ is odd, then
    $d_\delta - a_\delta \in \PP$.
    \end{itemize}
\end{lem}
\begin{proof}
 The hypothesis implies $a^2 \in 1+\PP^{\delta}$, and we deduce from  Lemma~\ref{L:squaring} that  $a,d\in 1+\PP^{\lceil \delta/2\rceil}$.  Therefore we may write  $a$ and $d$ as in \eqref{E:char2adexpression}. The product of $a$ and $d$ thus has only three pairs of terms that could contribute to the coefficient of $\varpi^\delta$ and we infer
$$
ad \equiv 1+ a^2_{\lceil\delta/2\rceil}\varpi^{2\lceil\delta/2\rceil}+
(a_{\delta} + d_{\delta})\varpi^\delta\mod \PP^{\delta + 1}.
$$
When $\delta$ is even, this implies that $d_\delta + a_\delta \in a_{\delta/2}^2 + \PP$. When $\delta$ is odd, $2\lceil \delta/2\rceil>\delta$ so that instead $a_\delta + d_\delta \in \PP$.
\end{proof}

We now derive the analogous result for  $2$-adic fields. Note that our statement specializes to Lemma~\ref{L:howcloseadareCHAR2} when we set $e=\infty$, so it is valid for all $F$.

\begin{prop}\label{L:howcloseadare}
Let $\delta \in \mathbb{Z}_{>0}$.   Suppose $a,d \in 1+\PP$ satisfy $a\equiv d \mod \PP^\delta$ and $ad \equiv 1 \mod \PP^{\delta+1}$.  Set $s=\max\{\delta-e,\lceil \delta/2\rceil\}$.  Then, replacing the pair $(a,d)$ by $(-a,-d)$ as necessary, we may write
\begin{equation}\label{E:char0adexpression}
a = 1 + \sum_{i\geq s}a_i\varpi^i, \quad \text{and}\quad
d = a + \sum_{i\geq \delta} (d_i-a_i) \varpi^i
\end{equation}
for some $a_i, d_i\in \RR$.  Moreover,
\begin{itemize}
    \item if $\delta\geq 2e+1$, then $a \in \pm 1 \mod \PP^{\delta-e}$ and $d_\delta - a_\delta \in \iota a_{\delta-e} +\PP$;
    
    \item if $\delta = 2e$, then $a\in 1+\PP^e$ and  $d_{2e}- a_{2e} \in \iota a_{e} + a_{e}^2 +\PP$;
    \item if $\delta < 2e$ is even, then $a\in 1+\PP^{\delta/2}$ and  $d_{\delta} - a_\delta\in a_{\delta/2}^2 + \PP$;
    \item if $\delta < 2e$ is odd, then $a\in 1+\PP^{\lceil \delta/2\rceil}$ and 
    $d_\delta - a_\delta \in \PP$.
    \end{itemize}
    In particular, if $q=2$ then in the case $\delta=2e$ we have simply $d_\delta - a_\delta \in  \PP$.
\end{prop}

\begin{proof}
The hypothesis implies that $a^2\equiv 1\mod \PP^\delta$, so by Lemma~\ref{L:squaring} we have $a\in \pm 1 + \PP^s$, where $s=\max\{\delta-e,\lceil \delta/2\rceil\}$.  Note that if the proposition is proven for a pair $(a,d)$, then it follows for the pair  $(-a,-d)$, so we may assume $a\in 1+\PP^s$.  Consequently,  $a,d$ may be written in the form given in \eqref{E:char0adexpression}.  The leading coefficient of $ad-1$ must therefore lie in $\PP^\delta$, and in fact we have
\begin{equation}\label{E:adterm}
ad \equiv 1+2a_s\varpi^s + a_s^2\varpi^{2s} + (a_\delta+d_\delta)\varpi^{\delta}\equiv 1+\iota a_s \varpi^{s+e}+ a_s^2\varpi^{2s}+(a_\delta+d_\delta)\varpi^{\delta}\mod \PP^{\delta+1}.
\end{equation}
By hypothesis,  this expression must be congruent to $1$.

If  $\delta\geq 2e+1$, then $s=\delta-e$ so that $2(\delta-e)\geq\delta+1$. Therefore we must have $\iota a_{\delta-e} + a_\delta + d_\delta \in \PP$, whence $d_\delta \equiv a_\delta +\iota a_{\delta-e} \mod \PP$, as required.

If $\delta=2e$, then $s=e$ and the three terms in \eqref{E:adterm} have valuation $\delta$.  We conclude that 
$\iota a_e+a_e^2 + a_{2e}+d_{2e} \in \PP$,
whence $d_\delta =d_{2e}\in \iota a_e+a_e^2 + a_{2e}+\PP$.

If $\delta<2e$ is even, then $\delta=2s$ but $s+e>\delta$, so there are two terms in \eqref{E:adterm} of minimal valuation and we require $a_{\delta/2}^2+a_\delta+d_\delta\in \PP$, 
yielding $d_\delta\in a_{\delta/2}^2+a_\delta\mod \PP$.

Finally, if $\delta<2e$ is odd, then $2s, s+e>\delta$.  Then the unique term of valuation  $\delta$ in \eqref{E:adterm} has coefficient $a_\delta+d_\delta \mod \PP$, whence $d_\delta \in a_\delta +\PP$.
\end{proof}

It is convenient to summarize the conclusions of Proposition~\ref{L:howcloseadare} as follows.

\begin{cor}\label{C: imageofaminusd}
    Let $\delta, \ell \in \mathbb{Z}_{>0}$ be such that $\delta <\ell$.  Consider the set $V$ of all pairs $(a,d)$ such that  $a,d \in 1+\PP$,  $a\equiv d \mod \PP^\delta$ and $ad \equiv 1 \mod \PP^\ell$.    Let $\rho$ be the map sending $(a,d)\in V$ to $(a - d)+\PP^{\delta+1}$.   
    Then for every $(a,d)\in V$,  $\rho(a,d)$ is independent of the choice of $d$.  Moreover, the image of $\rho$ in $\PP^\delta/\PP^{\delta+1}$ is represented by  
    \begin{itemize}
        \item $\PP^{\delta}$, if $\delta \geq 2e + 1$, or if $\delta < 2e$ is even;
        \item $\PP^{\delta+1}$, if $\delta < 2e$ is odd;
        \item $\mathscr{M}\varpi^\delta + \PP^{\delta+1}$, if $\delta = 2e$, where $\mathscr{M}\subset \mathscr{F}$ is the image of the map $x\mapsto \iota x + x^2$; in particular, $|\mathscr{M}| = q/2.$
    \end{itemize}
\end{cor}

\begin{proof}
    By Lemma~\ref{L:howcloseadareCHAR2} in characteristic two, and Proposition~\ref{L:howcloseadare} in general, the coefficient mod $\PP$ of $\varpi^\delta$ in $a-d$, equivalently, $d_\delta-a_\delta$, is entirely determined by the first nontrivial coefficient of $a$ (or of $-a$, when $a\in -1+\PP^{\delta-e}$ in the case that $\delta \geq 2e+1$).  Recall that the squaring map is an automorphism on $\mathfrak{f}$, and that when $\car(F)=2$ we always have $\delta<2e$.  Thus, gathering these cases of Proposition~\ref{L:howcloseadare} yields the result.
\end{proof}

\section{Depth-zero supercuspidal representations of \texorpdfstring{$G$}{G} and of \texorpdfstring{$G'$}{G'} } \label{S:4}

In this section, we summarize the construction of the depth-zero supercuspidal representations of $G=\mathrm{GL}(2,F)$ and their branching rules upon restriction to a maximal compact subgroup 
as computed in \cite{Han}.

\subsection{Depth-zero supercuspidal representations of \texorpdfstring{$\Gl(2,F)$}{GL(2,F)}}\label{SS:depthzero}

We begin by recapping the representation theory of 
the finite group $\mathsf{G}= \mathrm{GL}(2,\mathfrak{f})$, where $\mathfrak{f}=\mathbb{F}_q$ with $q=2^f$ for some $f\in \mathbb{N}$.  
Our notation follows that of \cite[Ch.15]{DJ}, which applies the theory of Deligne--Lusztig representations.

The group $\mathsf{G}$ has two conjugacy classes of maximal tori over $\mathbb{F}_q$.  Let $\mathsf{S}$ denote the  split diagonal torus, which has $(q-1)^2$ elements, and $\mathsf{T}$ a nonsplit torus, which has $q^2-1$ rational elements. 
All irreducible representations are obtained as the irreducible components of the Deligne--Lusztig induction of characters of these tori.

An element of $\mathsf{T}$ can be realized as the matrix over $\mathbb{F}_q$ representing multiplication in $\mathbb{F}_{q^2}^\times$, whence its set of eigenvalues is given by $\{x,x^q\}$ for some $x\in \mathbb{F}_{q^2}^\times$.  It follows that the conjugacy classes in $\mathsf{G}$ may be indexed as in the first row of Table~\ref{Table:Deligne}, which is the character table for $\mathsf{G}$ as reproduced from \cite[Ch.15, Table 1]{DJ}.  The second and third rows count the number of classes and their cardinalities.  Rows four through six correspond to representations obtained via Deligne--Lusztig theory (in this case, parabolic induction) from $\mathsf{S}$.  The final row corresponds to those irreducible representations obtained from $\mathsf{T}$ and these are the cuspidal representations.  There are $\frac12q(q-1)$ distinct cuspidal representations, each of degree $q-1$.  Moreover, as shown in \cite[Ch 15]{DJ}, when $q$ is \emph{even}, each of these cuspidal representations restricts irreducibly to  $\mathrm{SL}(2,\mathfrak{f})$, and these give all the cuspidal representations of this group.  (When $p$ is odd, there are two non-Deligne--Lusztig cuspidal representations of $\Sl(2,\mathfrak{f})$ of half the degree.)

\begin{table}[!htb]
    \centering
    \begin{tabular}{|c|c|c|c|c|}
    \hline
         Class &  $\begin{bmatrix}
             a\\ & a
         \end{bmatrix}, a\in\mathbb{F}_q^\times$ & $ \begin{bmatrix}
             a \\  & b
         \end{bmatrix}, \begin{array}{l}
         a,b\in \mathbb{F}_q^\times\\ a\not = b 
         \end{array}$ & $\begin{bmatrix}
             x &  \\  & x^q
         \end{bmatrix}, \begin{array}{l}
         x\in \mathbb{F}_{q^2}^\times\\ x\neq x^q 
         \end{array}$ & $\begin{bmatrix}
             a & 1\\ & a
         \end{bmatrix}, a\in\mathbb{F}_q^\times $\\ \hline \hline
         Number of Classes & $q-1$ & $(q-1)(q-2)/2$ & $q(q-1)/2$ & $q-1$\\
         Cardinality of Class & $1$ & $q^2+q$ & $q(q-1)$ & $q^2-1$\\ \hline \hline
         $R(\{\alpha,\beta\}), \alpha\neq\beta\in \widehat{\mathbb{F}_q^\times}$
         & $(q+1)\alpha(a)\beta(a)$ & $\alpha(a)\beta(b)+\alpha(b)\beta(a)$ &  0 & $\alpha(a)\beta(b)$\\
         $\gamma\circ \det, \gamma\in \widehat{\mathbb{F}_q^\times}$ & $\gamma(a^2)$ & $\gamma(ab)$ & $\gamma(x\cdot x^q)$ & $\gamma(a^2)$ \\ 
         $\gamma\circ\det \otimes \mathsf{St}, \gamma\in \widehat{\mathbb{F}_q^\times}$ & $q\gamma(a^2)$ & $\gamma(ab)$ & $-\gamma(x,x^q)$ & $0$\\ \hline
         $-R_{\mathsf{T}}^{\mathsf{G}}(\omega), \omega \in \widehat{\mathbb{F}_{q^2}^\times}, \omega\neq \omega^q$ & $(q-1)\omega(a)$ & 0 & $-\omega(x)-\omega(x^q)$ & $-\omega(a)$\\ \hline
    \end{tabular}
    \caption{Character table of $\mathsf{G}=\mathrm{GL}(2,\mathbb{F}_q)$, reproduced from \cite[Ch.15, Table 1]{DJ}.  Restricting to conjugacy classes in $\mathsf{G}'=\mathrm{SL}(2,\mathbb{F}_q)$ gives the character table for $\mathsf{G}'$ when $q$ is even.}
    \label{Table:Deligne}
\end{table}

By \cite[Proposition 6.6]{MoyPrasad1996}, which holds without restriction on the residual characteristic, all depth-zero supercuspidal representations of $G$ arise from cuspidal representations of $\mathsf{G}$, as follows.

The reductive quotient $\K/\K_+$ 
is isomorphic as an $\mathfrak{f}$-group to $\mathsf{G}$.  Therefore we may inflate a cuspidal representation of $\mathsf{G}$ to a representation $(\sigma,V_\sigma)$ of $\K$.  The normalizer of $\K$ in $G$ is $Z\K$, where $Z$ denotes the center of $G$ and this coincides with the stabilizer in $G$ of the image of $x_0$ in $\buil^{red}(G)$.  The irreducible extensions of $\sigma$ to $Z\K$ are parametrized by the characters $\chi$ of $Z$ extending the central character $\omega$ of $\sigma$, and the representation
\begin{equation}\label{E:depthzero}
\pi(\chi,\sigma) :=\cInd_{Z\K}^{G}\chi \otimes \sigma
\end{equation}
is an irreducible supercuspidal representation of $G$ of depth zero.  Moreover, all irreducible depth-zero supercuspidal representations  arise in this way, for different choices of $\chi$ and $\sigma$.

The following theorem is due to Hansen \cite{Han}, for any residual characteristic.

\begin{thm}\label{Thm:Hansendepthzero}
Let $\pi=\cInd_{Z\mathcal{K}}^{G} \chi \otimes \sigma$ be an irreducible supercuspidal representation of $G=\Gl(2,F)$ of depth $0$, where $\sigma$ is the inflation of a cuspidal representation to $\mathcal{K}$ and $\chi$ is a character of $Z$ extending the central character of $\sigma$. Then with $\gell := \diag(\varpi^\ell, 1)$ and $B_\ell$ the group of lower triangular matrices mod $\PP^\ell$, we have
$$
\Res_{\mathcal{K}} \pi \cong 
\sigma \oplus \bigoplus_{\ell\geq 1}
\Indd_{B_\ell}^{\mathcal{K}} {}^{\gell}\sigma.
$$ 
Moreover, every summand is irreducible and  independent of $\chi$.  When $\ell\geq 1$ the corresponding summand has degree $q^{\ell-1}(q^2-1)$ and depth $\ell$ as a representation of $\mathcal{K}$.
\end{thm}

\begin{proof}
That $\{g_\ell\mid \ell \geq 0\}$ is a set of coset representatives for $\mathcal{K}\backslash G/Z\mathcal{K}$ follows from the KAK decomposition.  Thus one has a Mackey decomposition with components of the form
$$
\Ind_{\K\cap \prescript{g_\ell}{}{(Z\K)}}^{\K} 
\prescript{g_\ell}{}{(\chi\otimes\sigma)}
$$
for each $\ell\geq 0$.  When $\ell=0$ the inducing subgroup is $B_0=\K$.  When $\ell>0$, ${}^{\gell}\mathcal{K}\cap \mathcal{K}=B_\ell$.  Since $\Res_{\K}(\chi\otimes \sigma)=\Res_{\K}\sigma$, the restriction is independent of $\chi$.  It is direct to show that $\mathcal{K}/\mathcal{K}_{\ell+1}$ is the smallest such quotient group through which $\Indd_{B_\ell}^{\mathcal{K}} {}^{\gell}\sigma$ factors, so the depth of that component is $\ell$.   The rest now follows as in \cite[Thm 2]{Han}.
\end{proof}

\subsection{Restriction to \texorpdfstring{$G'$}{G'}} The irreducible supercuspidal representations of depth zero of $G'$ are exactly the irreducible components of the restriction to $G'$ of some $\pi=\pi(\chi,\sigma) = \cInd_{Z\mathcal{K}}^{G} \chi\otimes\sigma$ as in Section~\ref{SS:depthzero}. 

\begin{lem}\label{dc4.1} 
A set of representatives for the double coset space $G'\backslash G/Z\K$ is $\{I, \gone=\diag(\varpi,1)\}.$
\end{lem}

\begin{proof}
The subgroup $G'Z\K$ is the inverse image of  $\mathcal{R}^\times(F^\times)^2$ under the determinant map and the quotient $F^\times/\mathcal{R}^\times(F^\times)^2$ is represented by $\{1,\varpi\}$. 
\end{proof}

Thus applying Mackey theory (Proposition~\ref{P:Mackey})  we have
\begin{align*}
    \Res_{G'} \pi(\chi,\sigma) = \Res_{G'} \cInd_{Z\mathcal{K}}^{G} \chi\otimes\sigma% \\&= \bigoplus_{g\in G\backslash {G}/Z\mathcal{K}} \Indd_{G\cap {}^{g}(Z\mathcal{K})}^{G} {}^{g}\chi\otimes\sigma \\  &
    =  \Indd_{G'\cap Z\mathcal{K}}^{G'} (\chi\otimes\sigma) \oplus \Indd_{G'\cap {}^{g_1} (Z\mathcal{K})}^{G'} {}^{g_1}(\chi\otimes\sigma).
\end{align*}
As $Z\K$ is the stabilizer of $x_0\in \buil^{red}(G)=\buil(G')$, we infer that $G'\cap Z\K=G'_{x_0}=\K'$, the stabilizer of $x_0$ in $G'$.  Recall that $x_1 = g_1\cdot x_0$ is an adjacent but non-$G'$-conjugate vertex whose stabilizer is the subgroup $\maxotherKSL= G'\cap\prescript{g_1}{}{(Z\mathcal{K})}$.

\begin{thm}\label{4.2}
The restriction of $\pi(\chi,\sigma)$ to $G'$ is the sum of two irreducible supercuspidal representations
\begin{equation} \label{E:Grestdepthzero}
\pi_0(\sigma) := \Indd_{\K'}^{G'}\sigma \quad \text{and}\quad
\pi_1(\sigma) := \Indd_{\maxotherKSL}^{G'} {}^{g_1}\sigma,
\end{equation}
one for each conjugacy class of maximal compact subgroup of $G'$ and these are independent of the choice of $\chi$.
Up to isomorphism all irreducible depth-zero supercuspidal representations of $G'$ arise in this way.
\end{thm}

\begin{proof}
    As mentioned in Section~\ref{SS:depthzero}, 
     the restrictions to $\mathrm{SL}(2,\mathfrak{f})$ of the cuspidal representations of $\mathrm{GL}(2,\mathfrak{f})$ are all irreducible and cuspidal. It follows that $\Res_{\maxKSL}\sigma$ and $\Res_{\maxotherKSL}({}^{g_1}\sigma)$ are each the inflation of a cuspidal  representation of the corresponding finite group quotient, which is isomorphic to $\mathrm{SL}(2,\mathfrak{f})$. 
     Since these maximal compact subgroups  are self-normalizing in $G'$, \cite[Proposition 6.6]{MoyPrasad1996} directly yields that as $\sigma$ varies over the cuspidal representations of $G$, $\pi_0(\sigma)$ and $\pi_1(\sigma)$ yield all irreducible supercuspidal representations of depth zero of $G'$.
\end{proof}

\subsection{Restriction to \texorpdfstring{$\mathcal{K}'$}{K'}}  \label{5dzero} 
As any two maximal compact subgroups of $G'$ are $G$-conjugate, we may recover the branching to any maximal compact subgroup from the restriction to $\K'$. 
From the preceding, we have two distinct ways to restrict a depth-zero 
supercuspidal representation of $G$ to $\mathcal{K}'$; our first step is to relate them.

Let $\pi=\pi(\sigma,\chi)$ be an irreducible depth-zero supercuspidal representation of $G$. Note that for any $\ell \geq 1$, $\mathcal{K}=\mathcal{K}'B_\ell$:  if $g\in \mathcal{K}$ we may choose any $b\in B_\ell$ for which $\det(g)=\det(b)$ and set $k=gb^{-1}\in \mathcal{K}'$.  Writing $B_\ell':= \mathcal{K}'\cap B_\ell$, it follows then by Mackey theory that 
$$
\Res_{\mathcal{K}'}\Ind_{B_\ell}^{\mathcal{K}}{}^{\gell}\sigma = \Ind_{B_\ell'}^{\mathcal{K}'}{}^{\gell}\sigma.
$$
Applying Theorem~\ref{Thm:Hansendepthzero}, we infer that these 
are the components of $\Res_{\mathcal{K'}}\pi(\sigma,\chi) = \Res_{\mathcal{K}'}(\pi_0(\sigma)\oplus \pi_1(\sigma))$, though now they  will not in general be irreducible.   Write $\sigma$ also for the inflation to $\K'$ of the restriction of $\sigma$ to $\Sl(2,\mathfrak{f})$.

\begin{cor}\label{Cor:MackeycomponentsdepthzeroSL2} 
    Let $\sigma$ be a cuspidal representation of $\mathrm{SL}(2,\mathfrak{f})$. Then
    $$
    \Res_{\mathcal{K}'}\pi_0(\sigma) \cong \sigma \oplus \bigoplus_{\ell \in 2\Z_{\geq 1}}\Ind_{B_\ell'}^{\mathcal{K}'}{}^{\gell}\sigma, 
    \quad \text{and} \quad 
    \Res_{\mathcal{K}'}\pi_1(\sigma) \cong \bigoplus_{\ell\in1+ 2\Z_{\geq 0}} \Ind_{B_\ell'}^{\mathcal{K}'}{}^{\gell}\sigma.    
    $$
    Each representation $\sigma(\ell):=\Ind_{B_\ell'}^{\mathcal{K}'}{}^{\gell}\sigma$ has degree $q^{\ell-1}(q^2-1)$ and decomposes as a direct sum of irreducible representations of depth $\ell$, all of the same degree.  
\end{cor}    

\begin{proof}
    By the Cartan decomposition, we have that a set of double coset representatives for either $\K'\backslash G'/\K'$ or $\K'\backslash G'/\maxotherKSL$ is $\{\xi_t=\diag(\varpi^{t}, \varpi^{-t})\mid t\in\mathbb{Z}_{\geq 0}\}.$  Setting $\ell=2t$, we have $\xi_t = \gell z^{-t}$ where $z=\diag(\varpi,\varpi)\in Z$, which implies that $\xi_t$ and $\gell$ act identically via conjugation.  Thus, applying the Mackey decomposition to each of the induced representations $\pi_i(\sigma)$ as in \eqref{E:Grestdepthzero} yields the first statement.  Since $\K'$ is normal in $\K$, the irreducible subrepresentations of each component are $\K$-conjugate by Theorem~\ref{T:Clifford}.  We  deduce the rest from Theorem~\ref{Thm:Hansendepthzero}. 
\end{proof}

Note that the supercuspidal representations denoted $\pi_1(\sigma)$ are in fact distinguished by the property that $\pi_1(\sigma)^{\K'_+}=\{0\}$, since $\K'=G'_{x_0}$ and $\pi_1(\sigma)$ is induced from a vertex that is not conjugate to $x_0$ \cite{Latham2017}.

We can summarize these results in the following diagram.
\begin{equation}\label{diagram}
    \begin{tikzcd}[column sep=large, row sep=large, every arrow/.append style={-latex, font=\normalsize}]
        \pi \arrow[r, "\Res_{\mathcal{K}}"] \arrow[d, "\Res_{G'}"] & 
        \displaystyle\bigoplus_{\ell\geq0} \text{Ind}_{B_\ell}^{\mathcal{K}}  {}^{\gell} \sigma \arrow[d, "\Res_{\mathcal{K}'}"] \\
       \pi_0(\sigma)\oplus \pi_1(\sigma) 
        \arrow[r, "\Res_{\mathcal{K}'}"] & 
         \displaystyle\bigoplus_{\ell \in 2\Z_{\geq 0}} \Ind_{B_\ell'}^{\mathcal{K}'}{}^{\gell}\sigma \, \oplus
      \,   \bigoplus_{\ell \in 1+2\Z_{\geq 0}} \Ind_{B_\ell'}^{\mathcal{K}'}{}^{\gell}\sigma
    \end{tikzcd}
\end{equation}

This holds also when $p$ is odd \cite[\S4]{Nev}.

\section{Intertwining operators of the Mackey components}\label{Sec: Branching Rules}

We next focus on each of the Mackey components
\begin{equation}\label{E:shorthandMackeycomponents}
\sigma(\ell):= \Ind_{B_\ell'}^{\mathcal{K}'}{}^{\gell}\sigma, \quad \text{where $g_\ell = \diag(\varpi^\ell,1)$,}
\end{equation}
for $\ell \geq 1$.  In this section, we compute the dimension of their self-intertwining space $\Sigma(\ell):=\dim\End_{\K'}(\sigma(\ell))$.
As a first step we require a 
set of representatives for the double coset space $B_\ell'\backslash \mathcal{K}'/B_\ell'$.

\begin{dfn}\label{D:Sellk} Let $\cS$ be a set of representatives for $\RR^\times / \RR^{\times 2}$, as in Lemma~\ref{Lemma:Cassel}. Let  $\cS_{k}$ denote a set of representatives of the equivalence classes of elements of $\cS$ modulo $\PP^k$.  For  any $1\leq k<\ell$, let $\cS_{\ell,k}$ denote a set of representatives for the equivalence classes of elements of $\cS$ modulo $\PP^{\min\{\ell-k, k\}}$. 
\end{dfn}

That is, $u,u'\in \RR^\times$ represent the same class in $\cS_k$ if $u'\in u (\RR^\times)^2(1+\PP^k)$.   
For example, if $F=\mathbb{Q}_2$ then with $\cS=\{1 + a_1\varpi + a_2\varpi^2\mid a_i \in \{0,1\}\}$ we have
\begin{itemize}
    \item $\cS_1=\cS_{\ell,1}=\cS_{\ell, \ell-1}=\{1\}$ for all $\ell\geq 2$;
    \item $\cS_2=\cS_{\ell,2}=\cS_{\ell, \ell-2} = \{1, 1+\varpi\}$ for all $\ell \geq 4$;
    \item $\cS_3=\cS_{\ell,3}=\cS_{\ell, \ell-3} = \cS$ for all $\ell \geq 6$.
\end{itemize}

\begin{lemma}\label{valueOfSellK}
    If $k>2e$ then $\cS_{k}=\cS$ and has cardinality $2q^e$;  otherwise, $\vert \cS_{k} \vert = q^{\lfloor k/2 \rfloor}$.  When $1\leq k\leq \ell/2$ we have $\cS_{\ell,k}=\cS_{\ell,\ell-k}=\cS_k$.  
\end{lemma}

\begin{proof}
    We choose $\cS$ as in Lemma~\ref{Lemma:Cassel}. If $k>2e$ (which occurs only when $\car(F)=0$) then all elements of $\cS$ are distinct modulo $\PP^k$.  For every odd $k<2e$, we have $\cS_{k}=\cS_{k-1}$.  For every even $k\leq 2e$, we may choose $k/2$ coefficients freely from $\mathscr{F}$, a set of representatives for the residue field. The final statement follows by the symmetry in $\cS_{\ell,k}$.   
\end{proof}

\begin{prop}\label{Prop:doublecosets6.2} 
For each $\ell\geq 1$ the double coset space $B_\ell'\backslash \mathcal{K}'/B_\ell'$ is represented by 
$$
\dcst_\ell := \{ I, w\} \cup \bigcup_{1\leq k <\ell} \left\{
g(k,\alpha):=\begin{bmatrix}1&\alpha \varpi^k\\0&1\end{bmatrix}\; \middle|\; \alpha \in \cS_{\ell,k}\right\}.
$$
\end{prop}

\begin{proof}
The group $\mathcal{K}'$ decomposes as the disjoint union of the following sets of matrices (of determinant one), each of which is invariant under left and right multiplication by elements of $B'_\ell$:
\begin{equation}\label{components}
{\K'} = \begin{bmatrix}\RR & \RR^\times\\\RR & \RR\end{bmatrix}\sqcup
\begin{bmatrix}\RR & \PP\smallsetminus \PP^2\\\RR & \RR\end{bmatrix}\sqcup
\cdots \sqcup
\begin{bmatrix}\RR & \PP^{\ell-1}\smallsetminus \PP^\ell\\ \RR & \RR\end{bmatrix}\sqcup B'_\ell.
\end{equation}
It follows from the Bruhat decomposition that the first set is equal to $B'wB'$, where $B'$ is the group of lower triangular matrices in $\mathcal{K}'$; working mod $B'_\ell$ we deduce it is the double coset $B'_\ell w B'_\ell$.  Similarly, the final set is the double coset represented by $I$.  Thus we are done if $\ell =1$.

Suppose now that $\ell\geq 2$.  Note that the remaining sets in \eqref{components} are the set differences  $B'_k\smallsetminus B'_{k+1}$ for each $1\leq k<\ell$.  
Let $g$ be an arbitrary element  of $B'_k\smallsetminus B'_{k+1}$.  It can be factored as 
$$
g = \begin{bmatrix}a&b\varpi^k \\ c&d\end{bmatrix} = \begin{bmatrix}a&0\\c&d-cba^{-1}\varpi^k\end{bmatrix}\begin{bmatrix}1&ba^{-1}\varpi^k\\0&1\end{bmatrix} \in B'_\ell \;g(k, ba^{-1})\;B'_\ell
$$
where $a,b,d\in \RR^\times$ and $c\in \RR$. If $h=\diag(u,u^{-1})$ for some $u\in \RR^\times$, then $hg(k,
\alpha)h^{-1}=g(k,u^2\alpha)$.  It follows that $g\in B_\ell' g(k,\alpha) B_\ell'$ for some $\alpha \in \cS$.
It remains to determine when two such elements yield the same double coset.  Suppose $\alpha,\alpha' \in \cS$ and there exist $h,h' 
\in B'_\ell$ such that $hg(k,\alpha)=g(k,\alpha')h'$.  Then modulo $\PP^\ell$ we have the matrix equality
$$
\begin{bmatrix}a&0\\c&d\end{bmatrix}\begin{bmatrix}1&\alpha\varpi^k\\0&1\end{bmatrix} \equiv
\begin{bmatrix}1&\alpha'\varpi^k\\0&1\end{bmatrix}\begin{bmatrix}a'&0\\c'&d'\end{bmatrix},
$$
(for some $a,d,a',d'\in \mathcal{R}^\times$, $c,c'\in \RR$), which yields
$$
a \equiv a'+ c'\alpha'\varpi^k, \quad a\alpha\varpi^k\equiv d'\alpha'\varpi^k, \quad c\equiv c', \quad d+c\alpha\varpi^k \equiv d',
$$
all modulo $\PP^\ell$.  Thus $a \alpha \equiv d'\alpha' \mod \PP^{\ell-k}$ and $d'\equiv d \mod \PP^k$.  Since $ad \equiv 1 \mod \PP^\ell$ we infer that
$$
a\alpha \equiv a^{-1}\alpha' \mod \PP^{\min\{k,\ell-k\}},
$$
implying $\alpha$ and $\alpha'$ are in the same equivalence class of $\cS$ modulo $\PP^{\min\{k,\ell-k\}}$.
It is direct to see that this necessary condition for equality of double cosets is also sufficient. 
\end{proof}

We now turn to the self-intertwining of our Mackey components $\sigma(\ell)=\Ind_{B_\ell'}^{\K'}\sigmaell$.  Applying Frobenius reciprocity and Mackey theory, one has
\begin{align}\label{E:MackeyFrobsigmaell}
    \notag\Hom_{\mathcal{K}'}(\Ind_{B'_{\ell}}^{\mathcal{K}'} \sigmaell,\Ind_{B'_{\ell}}^{\mathcal{K}'} \sigmaell) 
    &\cong \Hom_{B'_{\ell}} (\sigmaell, \Res_{B'_\ell}\Ind_{B'_{\ell}}^{\mathcal{K}'} \sigmaell) \\
    &\cong \bigoplus_{\gamma\in B'_\ell \backslash \mathcal{K}' / B'_\ell} 
    \Hom_{B'_{\ell}} (\sigmaell, \Ind_{\prescript{\gamma}{}{B_\ell'} \cap B'_\ell}^{B'_\ell} {}^\gamma (\sigmaell)) \\\notag
    &\cong \bigoplus_{\gamma\in B_\ell'\backslash \mathcal{K}'/B_\ell'} \Hom_{\prescript{\gamma}{}{B_\ell'} \cap B_\ell'}(\sigmaell, \prescript{\gamma g_\ell}{}{\sigma}). 
\end{align}

The dimensions of these spaces can be computed using characters.  Define $\chi(u) := \sum_{x\in \mathscr{F}^\times}\psi(xu)$ for the sum of the nontrivial additive characters of $\mathfrak{f}$, inflated to characters of $\RR$.

\begin{lemma}\label{L:charactersigmaell}
    For $\ell\geq 1$ the trace character $\chi_\ell$ of $\sigmaell$ is given on $a\in B_\ell'$ by
    $$
    \chi_\ell(a) = \chi_\ell\left(\begin{bmatrix} a_{11} &a_{12}\varpi^\ell\\ a_{21} & a_{22} \end{bmatrix}\right)
    =\begin{cases}
q-1 & \text{if $a_{11}\in 1+\PP$ and $a_{12}\in \PP$};\\
-1 & \text{if $a_{11}\in 1+\PP$ and $a_{12}\in\RR^\times$};\\
0 & \text{otherwise}.    
\end{cases}
$$
In particular, $\sigmaell$ is an irreducible representation of $B'_\ell$, but upon further restriction to the subgroup defined by $a_{11},a_{22}\in 1+\PP$, its character reduces as
$\chi_\ell(a) = \chi(a_{12})$. 
\end{lemma}

\begin{proof}
    Let $a=\begin{bmatrix} a_{11} &a_{12}\varpi^\ell\\ a_{21} & a_{22} \end{bmatrix}\in B_\ell'$.  Then $g_\ell^{-1}ag_\ell=\begin{bmatrix} a_{11} &a_{12}\\ a_{21}\varpi^\ell & a_{22} \end{bmatrix}\in \K'$ is upper triangular modulo $\K'_+$.  The values of $\chi_\ell(a) = \Tr(\sigma(g_\ell^{-1}ag_\ell))$ can now be read from Table~\ref{Table:Deligne}, noting that $\Sl(2,\mathfrak{f})$ has trivial center.  
    Writing temporarily $B$ for $B'_\ell/B'_\ell \cap \prescript{g_\ell}{}{\K_+}$, we compute 
$$
\dim(\Hom_{B'_\ell}(\sigmaell,\sigmaell))=\frac{1}{|B|}\sum_{g\in B}\chi_\sigma(g)\overline{\chi_\sigma(g)} = \frac{1}{(q-1)q}\left((q-1)^2+(q-1)\right)=1,
$$
whence $\sigmaell$ is an irreducible representation of $B'_\ell$.  For the final point, note that the unit upper-triangular subgroup of $\Sl(2,\mathfrak{f})$ is isomorphic to $\mathfrak{f}$ and for all $u\in \RR$
    $$
    \sum_{x\in \mathscr{F}^\times}\psi(x u) = \begin{cases}
q-1 & \text{if $u\in \PP$};\\
-1 & \text{if $u\in\RR^\times$}.
\end{cases}$$
\end{proof}

It follows from the independence of $\sigmaell$ of the choice of cuspidal representation $\sigma$ that for all $\ell\geq 1$, the Mackey components $\sigma(\ell)$ are also independent of the choice of $\sigma$.  This is an example of a general phenomenon analyzed in \cite{Nevins2014}.

Our key calculation is the following.

\begin{thm}\label{dimHomeSigmaEll}
    Suppose $\ell\geq 1$ and let $\gamma\in \dcst_\ell$ represent a double coset of $B_\ell'\backslash \mathcal{K}'/B_\ell'$. Then we have
$$
\dim(\Hom_{\prescript{\gamma}{}{B'_\ell}\cap B'_\ell}(\sigmaell, \prescript{\gamma g_\ell}{}{\sigma}) )
= \begin{cases}
1 & \text{if $\gamma = I$;}\\
q-1 & \text{if $\gamma=g(k,\alpha)$ where $\ell-k<2e$ is odd and $2k>\ell$;}\\
1 & \text{if $\gamma=g(k,\alpha)$ where $\ell - k = 2e$ and $2k>\ell$;}\\
0 & \text{otherwise}.
\end{cases}
$$
\end{thm}

\begin{proof}
    Let $\gamma \in \dcst_\ell$.  Once and for all, we write 
 \begin{equation}\label{form}
 a=\begin{bmatrix}
     a_{11} & a_{12}\varpi^\ell\\ a_{21} & a_{22}
 \end{bmatrix}
 \in \intersec,
 \end{equation}
for some  $a_{ij}\in\RR$,  to represent an arbitrary element of this intersection.  

When $\gamma = I$, the intertwining number is one, by Lemma~\ref{L:charactersigmaell}. 
If $\gamma = w$, then $D_\ell:={}^wB'_\ell \cap B'_\ell$ consists of matrices that are diagonal modulo $\PP^\ell$.  Thus for any $a\in D_\ell$ as in \eqref{form}, there is some $a'_{21}\in \mathcal{R}$ such that 
$a_{21}=a'_{21}\varpi^\ell$. We compute
$$
{}^w\sigmaell(a) = \sigma(\gell^{-1}w^{-1}aw\gell)=\sigma\left(\begin{bmatrix}
    a_{22} & -a'_{21}\\ -a_{12}\varpi^{2\ell} & a_{11}\end{bmatrix}\right) = \sigma\left(\begin{bmatrix}a_{22} & a'_{21}\\ 0 & a_{11}\end{bmatrix}\right)
$$
whose value is independent of $a_{12}$.  From Lemma~\ref{L:charactersigmaell} we may infer that the restrictions to $D_\ell$ of $\sigmaell$ and of $\prescript{wg_\ell}{}{\sigma}$ are each irreducible; since $\sigmaell$ varies with the value of $a_{12}$ and $\prescript{wg_\ell}{}{\sigma}$ does not, they cannot intertwine.  Thus $\Hom_{{}^wB'_\ell \cap B'_\ell}(\sigmaell,\prescript{wg_\ell}{}{\sigma})=\{0\}$.

It remains to consider double coset representatives of the form $\gamma=g(k,\alpha)$ for some $1\leq k < \ell$ and $\alpha \in \cS_{\ell,k}$.  
By this token, we compute, for $a$ as in \eqref{form}, that
\begin{equation}\label{genElementTwistBlopp}
    b:=\gamma a\gamma^{-1} = \begin{bmatrix}
        a_{11} + \alpha a_{21}\varpi^k & (a_{22}-a_{11})\alpha \varpi^{k} - \alpha^2a_{21}\varpi^{2k}+a_{12}\varpi^{\ell}\\
        a_{21}\varpi^\ell & a_{22}-\alpha a_{21}\varpi^k
    \end{bmatrix}. 
\end{equation}

Thus $b$ is an element of $B_\ell'$ if and only if $(a_{22}-a_{11})\alpha \varpi^{k} - \alpha^2a_{21}\varpi^{2k}\in \PP^\ell$.  Since $k\geq 1$ we infer $a_{11}\equiv a_{22} \mod \PP^{\min\{k,\ell-k\}}$ and since $\det(a)=1$ we must have $a_{11}\equiv a_{22}\mod \PP$.  
It follows that the image of
$\prescript{g_{\ell}^{-1}}{}{(\prescript{\gamma}{}{B_\ell'} \cap B_\ell')}$ 
in $\mathcal{K}'/\mathcal{K}'_+\cong \mathrm{SL}_2(\mathfrak{f})$ is contained in the unit upper triangular subgroup $U$.  

By Lemma~\ref{L:charactersigmaell} we have $\chi_\ell(a)=\chi(a_{12})$ whereas
\begin{equation}\label{E:charslisted}
\prescript{\gamma}{}{\chi_\ell}(a)=\chi_\ell(b)=\chi((a_{22}-a_{11})\alpha \varpi^{k-\ell} - \alpha^2a_{21}\varpi^{2k-\ell}+a_{12}).
\end{equation}

Suppose first that $k\leq \ell -k$, so that $2k\leq \ell$.  Then the matrix $b$ as in \eqref{genElementTwistBlopp} lies in $B_\ell'$ only if  $a_{11}\equiv a_{22}\mod \PP^{k}$, whence $a_{11}^2 \equiv a_{11}a_{22}\equiv 1 \mod \PP^k$. 
We claim that for each choice of triple $(u,a_{12},a_{11})$ such that $u\in \mathscr{F}$, $a_{12}\in\RR$ and $a_{11}\in 1+\PP$ such that $a_{11}^2\in 1+\PP^k$, there exist unique $a_{21}\in \RR$ and $a_{22} \in 1+\PP$ such that 
\begin{equation}\label{E:system}
    (a_{22} - a_{11})\alpha\varpi^k - \alpha^2 a_{21}\varpi^{2k} = u\varpi^\ell, \quad \text{and}\quad a_{11}a_{22} - a_{12}a_{21}\varpi^\ell = 1.
\end{equation}
Indeed, this is linear system in the variables $a_{22}$ and $a_{21}$, yielding the unique solution
$$
    a_{22}= \frac{a_{12}u\varpi^{2\ell}+a_{11}a_{12}\alpha \varpi^{k+\ell} - \alpha^2\varpi^{2k}}{a_{12}\alpha \varpi^{k+\ell}-a_{11}\alpha^2\varpi^{2k}}\quad \text{and}\quad
    a_{21}= \frac{a_{11}u\varpi^\ell + a_{11}^2\alpha \varpi^k - \alpha \varpi^k}{a_{12}\alpha \varpi^{k+\ell} - \alpha^2a_{11}\varpi^{2k}}.
$$
The first equation yields $a_{22}\in 1+\PP$ since $a_{11}\in 1+\PP$.  The second equation yields $a_{21}\in \RR$ if and only if $a_{11}^2-1\in \PP^{k}$. Thus the intersection $\prescript{\gamma}{}{B_\ell'} \cap B_\ell'$ is parametrized by these triples.
Since 
$$
\chi_\ell(a) = \chi(a_{12}) \quad \text{and} \quad \prescript{\gamma}{}{\chi_\ell}(a)=\chi_\ell(b)=\chi(u+a_{12}).
$$
and $u,a_{12}$ are independent, it follows that the inner product of these characters is $0$, yielding as above that $\Hom_{\prescript{\gamma}{}{B_\ell'} \cap B_\ell'}(\sigmaell,\prescript{\gamma}{}{\sigmaell})=\{0\}.$

We assume from now on that $k>\ell -k$, which is equivalent to  $2k>\ell$ and $\ell-k<\ell/2$.  Thus  for all $b=\gamma a \gamma^{-1}\in B_\ell'$ the expression \eqref{E:charslisted} simplifies to
\begin{equation}\label{E:chiellhardcase}
\prescript{\gamma}{}{\chi_\ell}(a)=\chi_\ell(b)=\chi((a_{22}-a_{11})\alpha \varpi^{k-\ell} +a_{12}).
\end{equation}
Define the subgroup
$$
\Gamma_{\ell,k} = \begin{bmatrix}
    1 + \PP^{\ell-k+1} & \PP^{\ell+1}\\
    \RR & 1 + \PP^{\ell-k+1}
\end{bmatrix}.
$$ 
Its intersection with $\intersec$ is a normal subgroup.  If we set $M = (\prescript{\gamma}{}{B_\ell'} \cap B_\ell')/(\Gamma_{\ell,k}\cap \prescript{\gamma}{}{B_\ell'} \cap B_\ell')$, then both $\sigmaell$ and $\prescript{\gamma g_\ell}{}{\sigma}$ factor through to representations of the finite group $M$.

Now the conditions on $a$ yielding $a\in M$, or equivalently, for the matrix $b$ as in \eqref{genElementTwistBlopp} to lie in  $B_\ell'$,  become  modulo $\PP^\ell$ the following \emph{quadratic} system of equations in the variables $a_{11}$ and $a_{22}$:
$$
a_{11}\equiv a_{22} \mod \PP^{\ell-k} \quad \text{and}\quad a_{11}a_{22}\equiv 1 \mod \PP^{\ell}.
$$
By Corollary~\ref{C: imageofaminusd}, applied with $\delta = \ell - k < \ell$, $a=a_{11}$ and $d=a_{22}$, elements of $M$ are parametrized by the independent pair of coefficients $(a_{11},a_{12})$.  Let us address each case in turn, in the order outlined in the corollary.

Suppose first that either $\ell-k \geq 2e+1$ or $\ell - k$ is even and strictly less than $2e$.  In either of these cases, Corollary~\ref{C: imageofaminusd} implies that the map $a_{11}\in 1+\PP^{\min\{e,\lceil \delta/2\rceil\}} \mapsto (a_{11}-a_{22})\alpha \varpi^{k-\ell}$ is surjective onto $\RR/\PP$.  As above, we conclude using \eqref{E:chiellhardcase} that the values of the characters $\prescript{\gamma}{}{\chi_\ell}(a)$ and $\chi_\ell(a)$ are independent, whence $\Hom_{\intersec}(\sigmaell, \prescript{\gamma g_\ell}{}{\sigma})=\{0\}.$

Suppose next that $\ell-k<2e$ and $\ell-k$ is odd.  By Proposition~\ref{L:howcloseadare}, we have that $a,d\in 1+\PP^{\lceil (\ell-k)/2 \rceil}$ and $a-d \in \PP^{\ell-k+1}$.  We infer that $\prescript{\gamma}{}{\chi_\ell}(a) = \chi_\ell(a)=\chi(a_{12})$
on $M$, so that the intertwining is
$$
\dim\Hom_{\intersec}(\sigmaell, \prescript{\gamma g_\ell}{}{\sigma}) =
\frac{1}{|M|}\sum_{a_{11},a_{12}}\chi(a_{12})^2 = \frac{1}{q}((q-1)^2+(q-1))=q-1.
$$

We finally proceed to the case that $\ell - k = 2e < k.$   By Corollary~\ref{C: imageofaminusd}, the map $a_{11}\in 1+\PP^{e} \mapsto (a_{11}-a_{22})+\PP$ has image equal to the subgroup $\mathscr{M}=\{\iota a_e+a_e^2 \mid a_e\in \RR/\PP\}$.  Note that in this case, the pair $(a_{11},a_{12})$ runs over the set $(1+\PP^e)/(1+\PP^{2e+1}) \times \RR/\PP$.
Write $\rho(x)=x^2+\iota x$ and let $a_e\in \mathscr{F}$ be shorthand to denote the coefficient of $\varpi^e$ in $a_{11}$ (mod $\PP$).  Then by \eqref{E:chiellhardcase} we have $\prescript{\gamma}{}{\chi_\ell}(a)=\chi(\rho(a_e)\alpha+a_{12})$ so that
\begin{align*}
    \dim\Hom_{\intersec}(\sigmaell, \prescript{\gamma g_\ell}{}{\sigma})&=\frac{1}{|M|}\sum_{\substack{a_{11}\in (1+\PP^e)/(1+\PP^{2e+1})\\ a_{12}\in \RR/\PP}}\chi(a_{12})\chi(a_{12}+\rho(a_e)\alpha)\\
    &= \frac{1}{q^2}\sum_{a_e,a_{12}\in \mathscr{F}} \chi(a_{12})\chi(a_{12}+\rho(a_e)\alpha).
\end{align*}
We compute the sum as follows.  When $a_e\in \ker(\rho)$, which is a subgroup of order $2$, we have
$$
\sum_{a_{12}\in \mathscr{F}} \chi(a_{12})\chi(a_{12}+0)= (q-1)^2+ (q-1) = q^2-q.
$$
For the remaining $q-2$ choices of $a_e$, there are two choices of $a_{12}$ for which one of the two terms in the sum is $q-1$ and the other is $-1$.  The remaining choices of $a_{12}$ give $(-1)^2$.  This yields
$$
\sum_{a_{12}\in \mathscr{F}} \chi(a_{12})\chi(a_{12}+\rho(a_e)\alpha) = 2(1-q)+(q-2)=-q.
$$
Thus altogether we have $\dim\Hom_{\intersec}(\sigmaell, \prescript{\gamma g_\ell}{}{\sigma})=\frac{1}{q^2}\left(2(q^2-q)+(q-2)(-q)\right) = 1$, as required.
\end{proof}

Setting $i=\ell-k$, we deduce that the set of double cosets supporting intertwining of $\sigma(\ell)$ is
\begin{equation}\label{E:dcstellsup}
\dcst_{\ell, sup}
=\{I\} 
\cup \{g(\ell-i,\alpha) \mid i<\ell/2, \; \text{ and either $i<2e$ is odd or $i=2e$, and } \alpha\in \cS_{i}\}.
\end{equation}

\begin{cor}\label{C:Sigmaell}
    Let $\ell \geq 1$.  
    Then 
    $$
    \dim \End_{\K'}(\sigma(\ell)) = |\cS_{\lceil \ell/2 \rceil}| = \begin{cases}
    q^{\lfloor (\ell+1)/4\rfloor} & \text{if $\ell\leq 4e$;}\\
    2q^e & \text{if $\ell \geq 4e+1$.}
    \end{cases}
$$
\end{cor}

\begin{proof}
Since $\lfloor (\lceil \ell/2\rceil)/2 \rfloor = \lfloor (\ell+1)/4\rfloor$ and  $\lceil \ell/2 \rceil > 2e$ if and only if $\ell\geq 4e+1$, the formula for $|\cS_{\lceil \ell/2 \rceil}|$ follows from Lemma~\ref{valueOfSellK}.
Now we compute $\Sigma(\ell):=\dim \End_{\K'}(\sigma(\ell))$.
If $\ell\in \{1,2\}$, then $\dcst_{\ell,sup}=\{I\}$ so $\Sigma(\ell)=1=q^{0}$, as required.

When $\ell\geq 3$,  Theorem~\ref{dimHomeSigmaEll} yields that $\dim \Hom_{\intersec}(\sigmaell,\prescript{\gamma g_\ell}{}{\sigma})=q-1$ for each  $\gamma=g(\ell-i,\alpha)$ such that $i$ is odd and satisfies $1\leq i <\min\{2e,\ell/2\}$, and $\alpha\in \cS_{\ell,\ell-i}=\cS_i$.  Additionally, if $\ell\geq 4e+1$, then with $\gamma=g(\ell-2e,\alpha)$, for any $\alpha \in \cS_{\ell,\ell-2e}=\cS_{2e}$, we have $\dim \Hom_{\intersec}(\sigmaell,\prescript{\gamma g_\ell}{}{\sigma})=1$.

Suppose first that $\ell>4e$, so that $2e<\ell/2$.  Using \eqref{E:MackeyFrobsigmaell} and Lemma~\ref{valueOfSellK} we compute 
\begin{align*}
\Sigma(\ell) &= \sum_{\gamma \in \dcst_\ell}\dim \Hom_{\intersec}(\sigmaell, \prescript{\gamma g_\ell}{}{\sigma})\\
%&= 1 + \sum_{\substack{1\leq i \leq 2e-1\\i \text{\;odd}}}(q-1)|\cS_{\ell,\ell-i}| + |\cS_{\ell,\ell-2e}|\\
&= 1 + \sum_{j=0}^{e-1}(q-1)|\cS_{\ell,2j+1}| + |\cS_{\ell,2e}|\\
&= 1+ \sum_{j=0}^{e-1}(q-1)q^j + q^e = 2q^e.
\end{align*}
In this case we have $e<\infty$ and we deduce that $\Sigma(\ell)=|\cS|=|\cS_{\lceil \ell/2\rceil}|$.

Now suppose that $\ell \leq 4e$, so that $\ell/2\leq 2e$.  The greatest odd integer strictly less than $\ell/2$ is $2z+1$ where $z=\lfloor (\ell+1)/4\rfloor-1$.  Thus we find as above that
$$
\Sigma(\ell) = 1 + \sum_{j=0}^z(q-1)|\cS_{2j+1}|= 1 + (q-1)\sum_{j=0}^zq^{j}= q^{\lfloor (\ell+1)/4\rfloor}
$$
as required.
\end{proof}   

In contrast, when $p$ is odd, the same strategy of proof specializes to show that $\dim(\End(\sigma(\ell)))=2$ for all Deligne--Lusztig cuspidal representations $\sigma$ and $\ell\geq 1$ \cite[\S5]{Nev}.  In that case the expression for the character $\chi_\ell$ is slightly more complex, as it depends on the central character of $\sigma$, but the double coset space $\dcst_\ell$ is much simpler since $\cS_{\ell,k}=\cS=\{1,\varepsilon\}$ for some nonsquare $\varepsilon\in \RR^\times$ for all $1\leq k <\ell$.

In Section~\ref{S:7}, we will realize the complete decomposition into irreducible subrepresentations of each $\sigma(\ell)$.

\section{Interlude: inferring an explicit multiplicity-free result when  \texorpdfstring{$q=2$}{q=2}}\label{S:6}
Corollary~\ref{C:Sigmaell} establishes the dimension of $\End_{\K'}\sigma(\ell)$ for each $\ell\geq 1$.  
When the residue field is $\mathbb{F}_2$, the inducing representation $\sigmaell=:\vartheta_\ell$ is a character and the representation space of $\sigma(\ell)$ is simply 
$$
\{h:\K'\to \mathbb{C}\mid h(bkk')=\vartheta(b)h(k) \forall b\in B_\ell', k'\in \K'_+, k\in \K'\}.
$$
In this section, we illustrate in this special case how to leverage the results of Section~\ref{Sec: Branching Rules} to prove that this algebra is abelian and hence that the decomposition is multiplicity-free.  

A restatement of Mackey theory is that the endomorphism algebra of self-intertwining operators on $\sigma(\ell)=\Ind_{B_\ell'}^{\K'}\vartheta_\ell$  is isomorphic to the Hecke algebra 
$$
\mathscr{H}=\mathscr{H}(B_\ell' \backslash {\K}' / B_\ell', \vartheta_\ell) = \left\{\mathcal{F}:{\K}'\to \mathbb{C} \mid \mathcal{F}(b_1kb_2) = \vartheta_\ell(b_1)\mathcal{F}(g)\vartheta_\ell(b_2)\, \forall b_1,b_2\in B_\ell', g\in {\K}'\right\},
$$
which is an algebra under convolution, denoted  $*$.
The isomorphism is given by sending $\mathcal{F}\in \mathscr{H}$ to the intertwining operator in $\End_{\K'}\sigma(\ell)$ given by $h \mapsto \mathcal{F}*h$ 
for all $h:\K'\to \mathbb{C} \in \Ind_{B_\ell'}^{\K'}\vartheta_\ell$.  The double cosets of $B_\ell' \backslash {\K}' / B_\ell'$ that support nonzero elements of $\mathscr{H}$ 
are precisely those parametrized by $\gamma$ for which $\Hom_{\intersec}(\vartheta_\ell,\prescript{\gamma}{}{\vartheta_\ell})\neq 0$, that is, 
for $\gamma$ in the set $\dcst_{\ell, sup}$ of \eqref{E:dcstellsup}.
For each such $\gamma$ let $\mathcal{F}_\gamma\in \mathscr{H}$ be the function supported on $B_\ell'\gamma B_\ell'$ such that $\mathcal{F}_\gamma(\gamma)=1$. 
Then $\{\mathcal{F}_\gamma \mid \gamma\in \dcst_{\ell,sup}\}$ is a basis for $\mathscr{H}$.
We wish to determine the action of these operators on a basis for the representation space of $\sigma(\ell)$.

\begin{lemma}\label{L:Cosetrepresentatives}
    A set of coset representatives for $B_\ell' \backslash \K'$ is $\Sigma :=\Sigma_0\cup \Sigma_w$ where
$$
\Sigma_0:=\left\{ u_\beta = \begin{bmatrix} 1 & \beta \\ 0 & 1\end{bmatrix} \;\middle|\; \beta\in \RR/\PP^{\ell}\right\}, \quad \text{and} \quad  \Sigma_w:=
\{ u_\beta w \mid \beta \in \PP/\PP^{\ell}\}
$$
and $w=\left[\begin{smallmatrix} 0&1\\-1&0\end{smallmatrix}\right]$ is the Weyl element.
\end{lemma}

\begin{proof}
It is a quick matrix calculation to deduce that these elements represent distinct cosets.  We compute $[\K':B_\ell']=[\K':B_1'][B_1':B'_\ell] = (q+1)q^{\ell-1}$, where the first term is the order of $\mathrm{SL}(2,\mathfrak{f})/B$ and the second is equal to the index of the corresponding quotient of Lie algebras. It follows that this set is complete.
\end{proof}

From the lemma we infer that a basis for the space of $\sigma(\ell)$ is the set $\{h_{a}\mid a\in \Sigma\}$ of functions supported on the right cosets $B_\ell'a$ and satisfying $h_a(a)=1$.

\begin{prop}\label{actionOnEndSpaceQ2}
    Fix a Haar measure on the compact group $\K'$. 
    Then $\mathcal{F}_I = \vol(B_\ell')I$ and for every $g(k,\alpha)\in \dcst_{\ell,sup}$ we have 
    $$
    \mathcal{F}_{g(k,\alpha)}*h_{u_\beta}\in \mathbb{R}h_{u_{\beta'}}, \quad \text{and} \quad 
    \mathcal{F}_{g(k,\alpha)}*h_{u_\beta w}\in \mathbb{R}h_{u_{\beta'}w}
    $$
    where $\beta'=\beta+\alpha\varpi^k$, for all $u_\beta\in \Sigma_0$ and $u_\beta w\in \Sigma_1$.
\end{prop}

\begin{proof} 
First consider $\gamma=I$. Then for any $a,a'\in \Sigma$ 
$$
(\mathcal{F}_I*h_a)(a') = \int_{{\K'}} \mathcal{F}_I(y)h_a(y^{-1}a')dy =\int_{B'_\ell} \vartheta_\ell(y)\vartheta_\ell(y^{-1})h_a(a')dy = \delta_{a,a'} \; \vol(B'_\ell).
$$
More generally, 
note that $(\mathcal{F}_\gamma*h_a)(a') =  0$ whenever $a'a^{-1} \notin B'_\ell \gamma B'_\ell$, since in this case both factors of the integrand are identically zero.
So let $\gamma = g(k, \alpha)\in \dcst_{\ell,sup}$. Since $k\geq 1$ any element of the double coset $B_\ell'\gamma B_\ell'$ is lower triangular modulo $\PP^k$ and thus its diagonal entries lie in $\RR^\times$. 
If $(a,a')\in (\Sigma_0,\Sigma_w)$ or $(a,a')\in (\Sigma_w,\Sigma_0)$, then since $w^{-1}=-w$, the product $a'a^{-1}$ takes the form $\pm u_{\beta'} w u_\beta$ for some $\beta,\beta'\in \RR$, at least one of which lies in $\PP$. 
We compute
$u_{\beta'} w u_\beta=\left[\begin{smallmatrix} -\beta' & 1-\beta \beta' \\ -1 & -\beta\end{smallmatrix}\right]$ 
and infer that at least one of its diagonal entries is not invertible, whence $a'a^{-1} \notin B_\ell'\gamma B_\ell'$.
On the other hand, if $a=u_\beta,a'=u_{\beta'}\in \Sigma_0$, or $a=u_\beta w,a'=u_{\beta'}w\in \Sigma_w$, then we have
$a'a^{-1} = \left[\begin{smallmatrix}1&\beta'-\beta\\ 0 & 1\end{smallmatrix}\right]$,
which lies in $B'_\ell g(k,\alpha) B'_\ell$ if any only if $\beta' - \beta \equiv \alpha\varpi^{k} \mod \PP^\ell$.  Since the values $\beta$ are distinct mod $\PP^\ell$, this implies that for each $a\in \Sigma$, there exists a \emph{unique} $a'\in \Sigma$ for which $(\mathcal{F}_{g(k,\alpha)}*h_a)(a')\neq 0$.  Hence $\mathcal{F}_{g(k,\alpha)}*h_a = \lambda h_{a'}$ for some scalar $\lambda$, which must be real since $\vartheta_\ell$ is real-valued.  The statement follows. 
\end{proof}

\begin{cor}\label{sigmaEllMultFree}
    When $\mathfrak{f}=\mathbb{F}_2$, the Mackey components $\sigma(\ell)$ are all multiplicity-free.  Consequently, they each decompose as a direct sum of $\Sigma(\ell)$ distinct irreducible subrepresentations.
\end{cor}

\begin{proof}
    From Proposition~\ref{actionOnEndSpaceQ2}, it follows that the actions of the operators 
$\mathcal{F}_{g(k,\alpha)}$ commute, up to potentially a scalar factor; thus  
for all $g(k,\alpha),g(k',\alpha')\in \dcst_{\ell,sup}$, the operator $
\mathcal{F}_{g(k,\alpha)}*\mathcal{F}_{g(k',\alpha')}-\mathcal{F}_{g(k',\alpha')}*\mathcal{F}_{g(k,\alpha)}
$
is diagonal with respect to the basis $\{h_a\mid a\in \Sigma\}$. Since $\End_{\K'}(\sigma(\ell))$ is isomorphic to a sum of matrix algebras, the subalgebra generated by its commutators is diagonal if and only if  all summands are of degree one.  
Thus $\End_{\K'}(\sigma(\ell))$ is in fact commutative, and the representation $\sigma(\ell)$ is multiplicity-free, whence the result.
\end{proof}

\section{Constructing representations from nilpotent orbits}\label{S:7}

From now onwards we again let $F$ be an arbitrary local nonarchimedean field of residual characteristic two.    We begin in Section~\ref{nilp} with some facts about nilpotent orbits in $\Sl(2,F)$, and then in Section~\ref{SS:Iuell} construct irreducible representations of $\K$ and $\K'$ starting from nilpotent elements of negative depth at $x_0$ in the corresponding Lie algebra.    In Section~\ref{SS:brules} we prove these are precisely the irreducible components of the restrictions to $\K$ and $\K'$ of the Mackey components $\sigma(\ell)$ and hence derive the branching rules for all irreducible depth-zero supercuspidal representations of $G'$.

\subsection{Nilpotent orbits in \texorpdfstring{$\mathfrak{sl}(2,F)$}{SL(2,F)}}\label{nilp} 
By Engel's theorem, 
any nilpotent element of $\mathfrak{g}'$ is $G'$-conjugate to a matrix of the form 
\begin{equation}\label{2.1}
X_v=\begin{bmatrix} 0 & 0\\ v & 0\end{bmatrix}.
\end{equation}
In fact (for any field $F$) these give a set of representatives for all the distinct nilpotent $G'$-orbits by choosing
\begin{equation}\label{nilporb}
v \in \{0\} \cup F^\times/(F^\times)^2.
\end{equation}
Thus when $\car(F)=0$,   Lemma~\ref{Lemma:Cassel} yields $4q^e+1$ nilpotent orbits in all, but when $\car(F)=2$, there are infinitely many.  For each $v$ 
write $\mathcal{O}_v$ for the $G'$-orbit of $X_v$.  
All nonzero orbits are principal, that is, maximal with respect to the closure ordering.  Note that all $G'$ conjugates of $X_v$ are of the form
\begin{equation} \label{Eq:conjXu}
{}^gX_v=v
\begin{bmatrix}
ab&-b^2\\a^2&-ab
\end{bmatrix}
\end{equation}
for some $a,b\in F$, not both zero.

Recall that the depth at $x$ of a nonzero element $X\in \g'$ is the unique $r\in \mathbb{R}$ such that $X\in \g'_{x,r}\smallsetminus \g'_{x,r+}$.
The following  lemma holds independent of $p$.

\begin{lem}\label{vertexnilpotent}
For each principal nilpotent orbit $\cO\subset \g'$ there exists a unique $G'$-orbit of points $x\in \buil(G')$ such that $\cO$ contains an element of depth zero at $x$.  In this case, $x$ is a vertex and $\cO$ contains elements of every even depth at $x$, whereas the nilpotent orbit $\varpi \cO$ contains elements of every odd depth at $x$. 
\end{lem}

\begin{proof}
Let $x\in\buil(G')$ and $X
\in \cO\cap (\g'_{x,0}\smallsetminus \g'_{x,0+})$ for some principal nilpotent orbit $\cO$.  Then for all $g\in G'$ we have $\Ad(g)X \in \cO\cap (\g'_{gx,0}\smallsetminus \g'_{gx,0+})$, so this condition is an invariant of the $G'$-orbit of $x$.  For each $v\in \mathcal{R}^\times/(\mathcal{R}^\times)^2$ we have $X_v \in \cO_v \cap (\g'_{x_0,0}\smallsetminus \g'_{x_0,0+})$ whereas $X_{v\varpi^{-1}} \in \cO_{v\varpi^{-1}} \cap (\g'_{x_1,0}\smallsetminus \g'_{x_1,0+})$.  

If $x$ is not a vertex, then $\g'_{x,0}/\g'_{x,0+}\cong \mathfrak{t}_0/\mathfrak{t}_{0+}$ for some split toral subalgebra $\mathfrak{t}$.  It follows that the elements of any nonzero coset have nonzero determinant, and thus $\cO \cap (\g'_{x,0}\smallsetminus \g'_{x,0+})=\emptyset$ for any nonzero nilpotent orbit $\cO$.  Thus $x$ is a vertex.  Suppose $v\in \cS$, our set of representatives for $\RR^\times/(\RR^\times)^2$.  
Using \eqref{Eq:conjXu} we deduce that $X_{a^2v}\in \cO$ for every $a\in F^\times$; these elements have even depth $2\val(a)$ at $x_0$ and have odd depth $2\val(a)+1$ at $x_1$.  The case $v\in \varpi\cS$ is analogous.  Since conjugation by $\K'$ preserves both the $G'$ orbit and the depth at $x_0$, and every nonzero nilpotent $\K'$-orbit contains some element $X_v$ with $v\in F^\times$, the result follows.
 \end{proof}

Let $\cO_v=G'\cdot X_v$.  Then by the Iwasawa decomposition $G'=\K'SU$ we have a further decomposition of $\mathcal{O}_v$  into disjoint $\K'$ orbits as
\begin{equation}\label{E:K'orbits}
    \mathcal{O}_v = \K'S\cdot X_v = \bigsqcup_{n\in\mathbb{Z}} \K'\cdot X_{v\varpi^{2n}},
\end{equation}
where the $\K'$-orbit of $X_{v\varpi^{2n}}$ consists of all elements of $\cO_v$ of depth $2n+\val(v)$ at $x_0$.

Recall that a \emph{degenerate coset} is a nonzero element of $\g'_{x,r}/\g'_{x,r+}$, for some $x\in \buil(G')$ and $r\in \mathbb{R}$, that contains a nilpotent element.  When $p$ is odd, every degenerate coset of $\g'$ meets a unique nilpotent orbit, and DeBacker proves in \cite{DeBacker2002Nilpotent} that the nilpotent orbits can be parametrized
by certain classes of pairs $(x,\xi)$ where $x\in\buil(G')$ and $\xi\in \g'_{x,0}/\g'_{x,0+}$ is a degenerate coset, equivalently, is the lift of a nilpotent element of the Lie algebra of $\mathsf{G}'_x=G'_{x,0}/G'_{x,0+}$.  This parametrization fails in an interesting way when $p=2$:
most orbits instead become ``close cousins" that cannot be distinguished in any depth-zero coset $\g'_{x,r}/\g'_{x,r+}$. We make this precise as follows. 

\begin{dfn}
    Let $s<t\in \mathbb{R}$.  
    Define a \emph{degenerate $(s,t)$ coset at $x$} to be a coset $X+\g'_{x,t} \in \g'_{x,s}/\g'_{x,t}$ where $X$ is a nilpotent element of depth $s$ at $x$, that is, $X\in\mathfrak{g}'_{x,s}\smallsetminus \mathfrak{g}'_{x,s+}$.  
\end{dfn}

When $\g'_{x,s+}=\g'_{x,t}$ we recover the notion of a degenerate coset.  At a vertex it suffices to consider integral $s,t$.

\begin{lem}\label{countNilpOrbits}
    Let $s<t\in \mathbb{Z}$ and let $\xi$ be a  $(s,t)$-degenerate coset at $x_0$.  If $\car(F)=2$ then $\xi$ meets infinitely many nilpotent $G'$-orbits whereas when $\car(F)=0$,
   \begin{itemize}
       \item if $t-s>2e$, then $\xi$ meets a unique nilpotent $G'$-orbit; 
       \item if $t-s=2e$, then $\xi$ meets exactly two nilpotent $G'$-orbits;
       \item for each $k\in \{0, 1, \ldots, e-1\}$, if $t-s\in \{2k,2k+1\}$, then $\xi$ meets exactly $2q^{e-k}$ nilpotent $G'$-orbits.
   \end{itemize} 
   More precisely, for any $F$ and for each $u\in \RR^\times$,  the set of nilpotent $G'$-orbits meeting $\xi=X_{u\varpi^{s}}+\g_{x_0,t}$ is $\{\cO_{u'\varpi^{s}}\mid u'\in\cS, u'\equiv u \in \cS_{t-s}\}$, where $\cS_r$ is as in Definition~\ref{D:Sellk}. 
\end{lem}

\begin{proof}   
    By \eqref{Eq:conjXu}, and the definition of the Moy--Prasad filtration of $\g'$ at $x_i$,  we infer that if some $G'$-conjugate of a nilpotent element $X_{u'}$ meets $\xi$, then $u'a^2 \in u+\mathcal{P}^{t-s}$ for some $a\in F$.  Since $t>s$ this forces $a\in \RR^\times$, so $u'$ must be in the square class of $u$ modulo $\PP^{t-s}$.   Thus $\g_{x_0,s}/\g_{x_0,t}$ partitions the set of nilpotent orbits into  equivalence classes indexed by $\cS_{t-s}$, and when $\g'$ has only finitely many nilpotent $G'$-orbits we can count the number of orbits in each class using  Lemma~\ref{Lemma:Cassel}.
\end{proof}

For example, when $F=\mathbb{Q}_2$, then the coset $\xi_r = X_u+\g'_{x_0,r}$ with $r>0$ and $u\in \cS$ satisfies:
\begin{itemize}
        \item if $r=1$, then $\xi_r$ meets $4$ nilpotent orbits;
        \item if $r=2$, then $\xi_r$ meets $2$ nilpotent orbits; and
        \item if $r\geq 3$, then $\xi_r$ meets a unique nilpotent orbit.
    \end{itemize}

\begin{dfn}
In the setting of Lemma~\ref{countNilpOrbits}, when $X$ and $Y$ are two nilpotent elements of depth $s$ at $x_0$ with $\K'$-conjugate degenerate $(s,t)$ cosets, that is, such that $X\in \K'\cdot Y+\g_{x_0,t}$, then we briefly say their $\K'$ orbits are \emph{equivalent modulo depth $t$.}      
\end{dfn}

This condition is equivalent to the $G'$-orbit of $X$ meeting the $(s,t)$ degenerate coset of $Y$ at $x_0$.

When $p$ is odd, equivalence modulo depth $t$ is simply $\K'$-conjugacy.  When $p=2$, in contrast, all nilpotent $\K'$-orbits consisting of elements of some fixed depth $s$ are equivalent modulo depth $s+1$.  For all $s\leq t\in \Z$, there are $|\mathcal{S}_{t-s}|$ distinct classes of $\K'$-orbits of depth $s$ with respect to equivalence modulo depth $t$, 
so for example, when $t-s\geq 3$ there are $4$ classes if $F=\mathbb{Q}_2$ but $2^{t-s-1}$ classes if $F=\Fpt$.

\begin{rem}
    Under $G=\Gl(2,F)$ there is only one nonzero nilpotent orbit in $\mathfrak{gl}(2,F)$, represented by $X_1$, and it is attached (in the sense of Lemma~\ref{vertexnilpotent}) to any vertex of the reduced building.  This orbit decomposes into $\K$ orbits as
    $$
    G\cdot X_1 = \bigsqcup_{n\in\mathbb{Z}} \K\cdot X_{\varpi^{n}}.
$$
\end{rem}

\subsection{Representations of \texorpdfstring{$K$}{K} from degenerate \texorpdfstring{$(-\ell,-\ell/2)$}{ell2} cosets at \texorpdfstring{$x_0$}{x0}}  \label{SS:Iuell}
By the preceding section, an arbitrary nilpotent  $\K$-orbit of $\g$ (or an arbitrary nilpotent $\K'$-orbit of $\g'$) of depth $-\ell$ at $x_0$ is represented by an element of the form $X_v$ with $v=u\varpi^{-\ell}$,  for some pair $(u,\ell)$ with $u\in \RR^\times$ and $\ell\in \mathbb{Z}_{>0}$.   (In fact, for $\K$ it suffices to take $u=1$ and for $\K'$ we may choose $u\in \cS$.) 
Using the Moy--Prasad isomorphism 
$$
\K_{\ell/2+}/\K_{\ell+}=G_{x_0,\ell/2+}/G_{x_0,\ell+}\to \g_{x_0,\ell/2+}/\g_{x_0,\ell+},
$$
we can construct a character $\eta_{(u,\ell)}$ of $\K_{\ell/2+}$ by the rule that for each $g\in \K_{\ell/2+}$, 
$$
\eta_{(u,\ell)}(g)=\psi(\Tr(X_{u\varpi^{-\ell}}(g-I))) = \psi(u\varpi^{-\ell}g_{12})
$$
where $g_{12}$ denotes the common $(1,2)$ entry of $g$ and $g-I$, modulo $\PP^{\ell+}$. %, as discussed in \eqref{E:MP}.  
Then $\eta_{(u,\ell)}$ depends only on the degenerate $(-\ell,-\ell/2)$ coset $X_{u\varpi^{-\ell}}+\g_{x_0,-\ell/2}$.

To simplify notation, we define
$$
m=\lceil \ell/2 \rceil, \quad \text{and} \quad m'=\lceil \ell/2+\rceil=\lfloor \ell/2\rfloor+1 \geq m;
$$ 
then in all cases we have $m+m'=\ell+1$. Note that $\eta_{(u,\ell)}$ is a character of $\K_{m'}$.
Set also $\eta'_{(u,\ell)}=\Res_{\K'_{\ell/2+}}\eta_{(u,\ell)}$, a character of $\K'_{m'}$.   

Our goal in this section is to produce an irreducible representation of $\K$ (respectively, of $\K'$) from such a character.
We begin with some Clifford theory.

\begin{lem}\label{normalizerCharEta}
Set  $Z_0=Z(G)\cap \K$ and $U_0 = U\cap \K$ for the subgroup of lower triangular unit matrices.  The normalizer in $\K$ of the character  $\eta_{(u,\ell)}$ of $\K_{m'}$ is $$N_{\K}(\eta_{(u,\ell)})=Z_0U_0\K_{m},$$ whereas the normalizer in $\K'$ of the character $\eta'_{(u,\ell)}$ of $\K'_{m'}$ is
$$
N_{\K'}(\eta_{(u,\ell)}) = \begin{cases}
U_0\K'_{\lceil m/2\rceil } & \text{if $\ell\leq 4e$;}\\
Z'U_0\K'_{m-e} & \text{if $\ell\geq 4e+1$.}\\
\end{cases}
$$
\end{lem}

\begin{rem}
    When $p$ is odd, the normalizer of the corresponding character in $\Sl(2,\RR)$ is significantly smaller, being $Z'U_0\K'_m$.
\end{rem}

\begin{proof}
  Recall that $m+m'=\ell+1$.  
An element $g\in \K$ satisfies $\eta_{(u,\ell)}={}^g\eta_{(u,\ell)}$ if and only if 
for all $W\in \g_{x_0,\ell/2+}=\g_{x_0,m'}$, we have
$\psi(u\varpi^{-\ell}W_{12})=\psi(u\varpi^{-\ell}(g^{-1}Wg)_{12})$, where $*_{12}$ denotes the $(1,2)$ entry of the corresponding matrix.  This is equivalent to the requirement that $(g^{-1}Wg)_{12}-W_{12}\in \PP^{\ell+1}$ for all such $W$.  

Write $g=(g_{ij})$ and set $\lambda=\det(g)\in \RR^\times$.  We require, for all $W=(W_{ij})$ with each $W_{ij}\in \PP^{m'}$, that
\begin{equation}\label{E:gwgw}
(g^{-1}Wg)_{12}-W_{12}=\lambda^{-1}g_{22}g_{12}(W_{11}-W_{22})-\lambda^{-1}g_{12}^2W_{21}+\lambda^{-1}g_{22}^2W_{12} - W_{12} \in \PP^{\ell+1}.
\end{equation}
Since this should hold for all $W$ we must have $g_{22}^2\equiv \lambda \mod \PP^m$ and thus $g_{22}\in\RR^\times$.  Since $W_{11}-W_{22}$ ranges freely over $\PP^{m'}$, we further require  $g_{12}\in \PP^m$, which itself guarantees $\lambda^{-1}g_{12}^2W_{21}\in \PP^{\ell+1}$.  Consequently $\det(g)=\lambda \equiv g_{11}g_{22}\mod \PP^{m}$ so that by the first observation $g_{11}\equiv g_{22}\mod \PP^m$.  Since $g_{21}$ may range freely over $\RR$ we conclude that $N_{\K}(\eta_{(u,\ell)})=Z_0U_0\K_{m}$, as required.

Now consider the normalizer of $\eta'_{(u,\ell)}$ in $\K'$.
In this case, $\lambda = \det(g)=1$.  Furthermore, an arbitrary element $W\in \g'_{x_0,m'}$ satisfies $W_{22}=-W_{11}$, so that the expression in \eqref{E:gwgw} simplifies instead to
$$
(g^{-1}Wg)_{12}-W_{12}= 2g_{22}g_{12}W_{11}-g_{12}^2W_{21}+(g_{22}^2 - 1)W_{12}.
$$
This lies in $\PP^{\ell+1}$ for all choices of $W$ if and only if $2g_{22}g_{12}$, $g_{12}^2$ and $g_{22}^2-1\in \PP^{m}.$

This last condition implies by Lemma~\ref{L:squaring} that  $g_{22}\in \pm 1+\PP^{\max\{m-e,\lceil m/2 \rceil\}}$. Since $g_{22}$ is invertible,  the first two conditions together imply $g_{12}\in \PP^{\max\{m-e,\lceil m/2 \rceil\}}$, 
recalling that this simplifies to $\PP^{\lceil m/2 \rceil}$ when $2=0$ or $e=\infty$.  Again, the element $g_{21}$ varies over $\mathcal{R}$.

Note that $m-e= \lceil m/2\rceil$ if and only if $m\in \{2e,2e+1\}$; unraveling this condition yields $4e-1\leq \ell\leq 4e+2$ so we may divide the cases at $\ell=4e$.  When $\ell\leq 4e$, we have $-1\in 1+\PP^{\lceil m/2\rceil}$, yielding the statement.
\end{proof}

One can show that $\eta_{(u,\ell)}$ and $\eta'_{(u,\ell)}$ do not extend to characters of their (large) normalizers.   To produce the required irreducible representation, we instead
first extend to a small intermediate subgroup. 

We begin with $\K=\Gl(2,\RR)$.  Define the groups
    $$
    \Jtype := \begin{bmatrix}
        1+\PP^{m} & \PP^{m'}\\ \PP^{m'} & 1+\PP^{m}
    \end{bmatrix}\cap \K \quad \text{and}\quad \intgroup := Z_0U_0\Jtype. 
    $$
They satisfy
$$
\K_{\ell/2+}\subseteq \Jtype \subset \intgroup \subset N_{\K}(\eta_{(u,\ell)}).
$$
An arbitrary element of $\intgroup$ is of the form $g=(g_{ij})\in \K$ such that $g_{12}\in \PP^{m'}$ and  $g_{11}-g_{22}\in \PP^{m}$. 

Since $m+m'=\ell+1$, it is straightforward to verify that the character  $\eta_{(u,\ell)}$ extends to a well-defined character of $\Jtype$ by the formula $\eta_{(u,\ell)}(k)=\psi(u\varpi^{-\ell}k_{12})$, where $k_{12}$ is the $(1,2)$ entry of $k$. 
To extend this further to a character of $\intgroup$, let $\zeta$ denote a character of $\RR^\times$ of depth less than $m$.
Since for any $z\in Z_0$, $c\in U_0$ and $k\in \Jtype$, the upper triangular entry of $g=zck\in \intgroup$ is $z_{11}k_{12} \mod \PP^{m'}$, the formula
\begin{equation}\label{E:hatetazetauell}
\hat{\eta}_{\zeta,(u,\ell)}(g)=\zeta(g_{11})\psi(u\varpi^{-\ell}g_{11}^{-1}g_{12})
\end{equation}
is a well-defined character of $\intgroup$ that restricts to $\eta_{(u,\ell)}$ on $K_{m'}$.
As $g_{12}\in \PP^{m'}$, the character 
$\hat{\eta}_{\zeta,(u,\ell)}$ 
depends only on $\zeta$, $\ell$ and the coset $u + \PP^{m}$ (equivalently, on $\zeta$ and the degenerate $(-\ell,-\ell/2)$ coset $X_{u\varpi^{-\ell}}+\g_{x_0,-\ell/2}$).  We define
\begin{equation}\label{E:defJzetauell}
J(\zeta,(u,\ell)) = \Ind_{\intgroup}^{\K}\hat{\eta}_{\zeta,(u,\ell)}.
\end{equation}
When $u=1$ we write $\hat{\eta}_{\zeta,\ell}:=\hat{\eta}_{\zeta,(1,\ell)}$  and $J(\zeta,\ell):=J(\zeta,(1,\ell))$.
    Note that if $g=\diag(1,u)\in \K$ then $\prescript{g}{}{\hat{\eta}_{\zeta,(1,\ell)}}=\hat{\eta}_{\zeta,(u,\ell)}$ and $g$ normalizes $\intgroup$.  Thus for all $u\in \RR^\times$ we have $J(\zeta,u,\ell)\cong J(\zeta,\ell)$.

Now let $\K'=\Sl(2,\RR)$.  
We first observe that $\intgroup\cap \K'$ is  strictly larger than $\{\pm I\}U_0(\K_{m,m'}\cap \K')$.

\begin{lem}
    The subgroup $\intgroup':= Z_0U_0\Jtype \cap \K'=Z_0U_0\Jtype\cap N_{\K'}(\eta_{(u,\ell)})$ is given by
    $$
    \intgroup' = \left\{ g=\begin{bmatrix} z+a\varpi^m & b\varpi^{m'}\\ u &z+d\varpi^m\end{bmatrix} \;\middle|\; a,b,d,u\in \RR^\times, z\in \pm 1+ \PP^{\max\{m-e,\lceil m/2 \rceil\}}, \det(g)=1\right\}.
    $$
\end{lem}

\begin{proof}
    The elements of $\intgroup$ are characterized as those $g=(g_{ij})\in \K$ such that $g_{11}-g_{22}\in \PP^m$ and $g_{12}\in \PP^{m'}$. When $\det(g)=1$, this implies $g_{11}^2\equiv 1 \mod \PP^m$, whence by Lemma~\ref{L:squaring} we have $g_{11}\in \pm 1+\PP^{\max\{m-e,\lceil m/2\rceil\}}$.  Using 
      Lemma~\ref{normalizerCharEta} we infer $\intgroup'\subset N_{\K'}(\eta_{(u,\ell)})$.
\end{proof}

We write $\hat{\eta}'_{\zeta,(u,\ell)}$ for the restriction of $\hat{\eta}_{\zeta,(u,\ell)}$ to $\intgroup'$, understanding that  $\hat{\eta}'_{\zeta,(u,\ell)}$ depends only on the restriction of  $\zeta$ to the subgroup $(\pm 1 + \PP^{\max\{m-e,\lceil m/2\rceil\}})/(1+\PP^m)$, where it is by Lemma~\ref{L:squaring} a quadratic character.  Then we may similarly define
\begin{equation}\label{E:defIzetauell}
I(\zeta,u,\ell) = \Ind_{\intgroup'}^{\K'}\hat{\eta}'_{\zeta,(u,\ell)}.
\end{equation}

\begin{thm}\label{T:JzetaellIzetauell}
    Let $\ell\in \mathbb{Z}_{\geq 1}$ and let $\zeta$ be a character of $\RR^\times$ of depth less than $m=\lceil \ell/2 \rceil$.  Then  the representations $J(\zeta,\ell)$ and $I(\zeta,u,\ell)$, for any $u\in \RR^\times$,
    are irreducible.
Moreover, $I(\zeta,u,\ell)\cong I(\zeta,u',\ell)$ if and only if $u$ and $u'$ represent the same class of squares modulo $\PP^m$, that is, $u\equiv u' \in \cS_m$; equivalently, if and only if the $\K'$ orbits of $X_{u\varpi^{-\ell}}$ and $X_{u'\varpi^{-\ell}}$ are equivalent modulo depth $-\ell/2$. In particular, there are only finitely many representations $I(\zeta, u, \ell)$ for each pair $(\zeta, \ell)$.
\end{thm}

\begin{proof}
When $\ell=1$, we see from Lemma~\ref{normalizerCharEta} that $\Gamma(1)=N_{\K}(\eta_{(u,\ell)})$ (and the same for $\K'$) so the representations are irreducible by Clifford theory.  Since $m=1$ and $\cS_1=\{1\}$, the remaining statements are automatic.

Suppose now $\ell\geq 2$.  We proceed as in the proof of Theorem~\ref{dimHomeSigmaEll}, noting that this case is simpler since  characters intertwine on a subgroup if and only if they are equal.  Recall that $m'=\lceil \ell/2+\rceil$ so that $m+m'=\ell+1$.
Let $u\in \RR^\times$. Applying Frobenius reciprocity and Mackey's theorem, we find 
    $$ 
    \Hom_{\K'}(I(\zeta,u,\ell), I(\zeta,u',\ell))\cong \bigoplus_{g\in \intgroup'\backslash \K'/ \intgroup'}\Hom_{\intgroup'\cap\prescript{g}{}{\intgroup'}}(\hat{\eta}_{\zeta,(u,\ell)}, \prescript{g}{}{\hat{\eta}'_{\zeta,(u',\ell)}}),
    $$
    and analogously for $\K$.
    Following a similar strategy to the proof of Proposition~\ref{Prop:doublecosets6.2} yields that a set of representatives for 
     $\intgroup\backslash \K/\intgroup$ is  
   $$
   \{w\} \sqcup\left\{ \gamma_{a,b}:=\begin{bmatrix}
        a^{-1} & b\\ 0 & 1
    \end{bmatrix}:\, a \in \RR^\times/(1+\PP^{m}),  b \in \RR/\PP^{m'} \right\},
    $$  
    whereas a set of representatives for $\intgroup'\backslash \K' /\intgroup'$ is  
   $$
   \{w\} \sqcup\left\{ \gamma'_{a,b}:=\begin{bmatrix}
        a^{-1} & b\\ 0 & a
    \end{bmatrix}:\, a \in \RR^\times/(\pm 1+\PP^{\max\{m-e,\lceil m/2 \rceil\}}),  b \in \RR/\PP^{m'} \right\}.
    $$ 
As in the proof of Theorem~\ref{dimHomeSigmaEll}, a double coset corresponding to $w$ cannot support an intertwining operator since $\hat{\eta}_{\zeta,(u,\ell)}(g)$ depends on the upper triangular entry of $g$ while  $\prescript{w}{}{\hat{\eta}_{\zeta,(u',\ell)}}$ depends on the independent lower triangular entry. 

We consider first the case of $\K=\Gl(2,\RR)$.  Since $J(\zeta,u,\ell)\cong J(\zeta,\ell)$ we assume $u=u'=1$.  Consider a double coset parametrized by $\gamma_{a,b}$, for some $a\in \RR^\times$, $b\in \RR$.  Let $g\in \intgroup$ be arbitrary; then 
$g_{12}\in \PP^{m'}$, $g_{21}\in \RR$ and $g_{11}-g_{22}\in \PP^{m}$.    
Such an element $g$ lies in $\intgroup\cap \prescript{\gamma_{a,b}}{}{\intgroup}$ if and only if
$$
\gamma_{a,b}^{-1}g\gamma_{a,b} = 
\begin{bmatrix}
    g_{11}-bg_{21} & ag_{12} + ab\left((g_{11}-g_{22})-bg_{21}\right)\\
    a^{-1}g_{21} & g_{22}
\end{bmatrix} \in \intgroup.
$$
If $b\notin \PP^{m'}$, then we may choose $g\in\intgroup$ so that $bg_{21}\in\PP^m$ and $b^2g_{21}\in \PP^{m'}\smallsetminus \PP^{\ell+1}$.  For such $g$, we have 
$$
\prescript{\gamma_{a,b}}{}{\hat{\eta}_{\zeta,\ell}}(g)=\zeta(g_{11}-bg_{21})\psi(\varpi^{-\ell}\left(ag_{12} + ab\left((g_{11}-g_{22})-bg_{21}\right)\right)),
$$
which depends on $g_{21}$ and hence cannot equal $\hat{\eta}_{\zeta,\ell}(g)$ (which does not).  On the other hand, if $b\in \PP^{m'}$ then the coset supports intertwining if and only if 
$ag_{12}\varpi^{-\ell} \equiv g_{12}\varpi^{-\ell} \mod \PP$ 
for all $g_{12}\in \PP^{m'}$.  Since $m+m'=\ell+1$ this happens if and only if $a \equiv 1 \mod \PP^m$.  Thus only the trivial double coset supports intertwining and $J(\zeta,\ell)$ is irreducible.

Now consider the case of $\K'=\Sl(2,\RR)$, and $u,u'\in \RR^\times$.   
Then an arbitrary element $g=(g_{ij})\in \intgroup'$ satisfies  $g_{11}\in \pm 1+\PP^m$, $g_{12}\in \PP^{m'}$ and $g_{11}g_{22}-g_{12}g_{21}=1$. For any $a\in \RR^\times$, $b\in \RR$, we have that such a $g$ lies in  $\intgroup\cap \prescript{\gamma'_{a,b}}{}{\intgroup'}$ if and only if
$$
{\gamma_{a,b}'}^{-1}g\gamma'_{a,b} = 
\begin{bmatrix}
    g_{11}-a^{-1}bg_{21} 
    & a^2g_{12} + 
    b\left(a(g_{11}-g_{22})-
    bg_{21}\right)\\
        a^{-2}g_{21} & 
        g_{22}+ba^{-1}g_{21}
\end{bmatrix} \in \intgroup'.
$$
As above, if $b\notin \PP^{m'}$, then we may choose $g$ so that $bg_{21}\in\PP^m$ and $b^2g_{21}\in \PP^{m'}\smallsetminus \PP^{\ell+1}$ (and $g_{11}-g_{22} \in \PP^{m'}$); for such $g$,  $\prescript{\gamma_{a,b}'}{}{\hat{\eta}_{(u',\ell)}}(g)$ depends on $g_{21}$ so the double coset supports no intertwining.  On the other hand, diagonal double cosets of the form $\gamma'_{a,0}$ support intertwining if and only if
$a^2g_{12}\varpi^{-\ell}u' \equiv ug_{12}\varpi^{-\ell} \mod \PP$ 
for all $g_{12}\in \PP^{m'}$, in other words, if and only if $a^2u' \equiv u \mod \PP^m$.  This happens if and only if $u$ and $u'$ represent the same square class modulo $\PP^m$.  Thus we may assume $u=u'\in \cS_m$, so that $a^2\equiv 1\mod \PP^m$, which by Lemma~\ref{L:squaring} is equivalent to $a \in \pm 1 + \PP^{\max\{m-e,\lceil m/2\rceil\}}$.  We conclude again that only the trivial double coset supports intertwining, and the representations $I(\zeta,u,\ell)$ are distinct and irreducible as $u$ ranges over $\cS_m$.  The final statements follows from   Lemma~\ref{countNilpOrbits} with $s=-\ell$ and $t=\lceil -\ell/2\rceil$, where $t-s=m$. 
\end{proof}

\subsection{The decomposition of \texorpdfstring{$\sigma(\ell)$}{sigmaell}, 
and the branching rules of \texorpdfstring{$\pi$}{pi}}\label{SS:brules}

Now suppose $\sigma$ is a cuspidal representation of $\Gl(2,\mathfrak{f})$; it is the Deligne--Lusztig induction from a character $\omega$ of an elliptic torus. 
The restriction of this character to the center inflates to a character of $Z_0$, and we thus identify it with a depth-zero character, also denoted $\omega$, of $\RR^\times$.  Since it has depth zero, it is trivial on $\pm 1+\PP^{\max\{m-e,\lceil m/2 \rceil\}}$. 

\begin{prop}\label{P:Jzetaellismackey}
    If $\sigma$ is the inflation of a cuspidal representation of $\Gl(2,\mathfrak{f})$ with central character $\omega$, then $J(\omega,\ell)\cong \Ind_{B_\ell}^{\K}\sigmaell$, the depth $\ell$ irreducible Mackey component  of Theorem~\ref{Thm:Hansendepthzero}. 
\end{prop}

\begin{proof}
Since $\omega$ has depth zero, $J(\omega,\ell)$ is well-defined for all $\ell>0$.   
Recall that $g_\ell = \diag(\varpi^\ell,1)$ and $B_\ell$ consists of matrices that are lower triangular modulo $\PP^\ell$. %, as in Theorem~\ref{Thm:Hansendepthzero}. 
To prove that $J(\omega,\ell)\cong \Ind_{B_\ell}^{\K}\sigmaell$ it suffices to show that these irreducible representations intertwine; by  
 Frobenius reciprocity and Mackey theory it suffices to show that the intertwining number
 $\dim\Hom_{\intgroup\cap B_\ell}(\hat{\eta}_{\omega,\ell},\sigmaell)$ is nonzero.

Since $\intgroup\cap B_\ell=\{g=(g_{ij})\in\intgroup \mid g_{12}\in \PP^\ell\}$, we may write $g_{12}=g'_{12}\varpi^\ell$ for some $g'_{12}\in \RR$, yielding
$$
\hat{\eta}_{\omega,\ell}(g)=\omega(g_{11})\psi(g_{11}^{-1}g'_{12}).
$$
On the other hand, recalling that  $g_{11}\equiv g_{22}\mod \PP^m$ and that $\sigma$ has depth zero, we find using Table~\ref{Table:Deligne} that the character of $\sigmaell$ is given 
by
$$
\Tr(\sigmaell(g)) = \Tr(\sigma\left(\begin{bmatrix}g_{11}&g'_{12}\\0&g_{11}\end{bmatrix}\right))=\begin{cases}
    (q-1)\omega(g_{11}) & \text{if $g'_{12}\in \PP$};\\
    -\omega(g_{11}) & \text{if $g'_{12}\in \RR^\times$.}
\end{cases}
$$
We now compute the intertwining of these characters.  Since both depend only on the values of $g_{11}\in \RR^\times/(1+\PP)$ and $g'_{12}\in \RR/\PP$, we compute
\begin{align*}
\dim\Hom_{\intgroup\cap B_\ell}(\hat{\eta}_{\omega,\ell},\sigmaell) &= 
\frac{1}{q(q-1)}\left(
\sum_{g_{11}\in \mathfrak{f}^\times}
\omega(g_{11})\overline{(q-1)\omega(g_{11})}
+
\sum_{g_{11}\in \mathfrak{f}^\times, g'_{12}\in \mathfrak{f}^\times}
\omega(g_{11})\psi(g_{11}^{-1}g'_{12})\overline{(-\omega(g_{11}))}
\right)
\\
%&= \frac{1}{q(q-1)}\left(-\sum_{g_{11}\in \mathfrak{f}^\times, k\in \mathfrak{f}^\times}
%\omega(g_{11})\psi(k)\overline{\omega(g_{11})}+(q-1)\sum_{g_{11}\in \mathfrak{f}^\times}
%\omega(g_{11})\overline{\omega(g_{11})}\right)\\
&= \frac{1}{q(q-1)}\sum_{g_{11}\in \mathfrak{f}^\times}\omega(g_{11})\overline{\omega(g_{11})}\left( q-1-\sum_{k\in \mathfrak{f}^\times}\psi(k)\right) = 1
\end{align*}
since $\omega$ is a character of $\mathfrak{f}^\times$ and $\psi$ is a nontrivial character of $\mathfrak{f}$.
\end{proof}

In particular, it follows that the degree of $J(\zeta,\ell)$ is $q^{\ell-1}(q^2-1)$ for all choices of $\zeta$, a fact one could compute directly.

\begin{cor}\label{C:Izetauell}
    Let $\ell\in \mathbb{Z}_{\geq 1}$ and let $\sigma(\ell)$ denote a Mackey component of a depth-zero supercuspidal representation of $G'$.    Then
    $$
    \sigma(\ell) \cong \bigoplus_{u\in \cS_{\lceil \ell/2 \rceil}} I(\triv, u,\ell)
    $$
    is the decomposition of $\sigma(\ell)$ into distinct irreducible subrepresentations.
\end{cor}

\begin{proof}
Set $m=\lceil \ell/2 \rceil$. 
Recall that for $\K'=\Sl(2,\RR)$, and any $\ell\geq 1$, the Mackey component $\sigma(\ell)$ of \eqref{E:shorthandMackeycomponents} 
satisfies
$$
\sigma(\ell) = \Res_{\K'}\Ind_{B_\ell}^{\K}\sigmaell = \Ind_{B_\ell'}^{\K'}\sigmaell. 
$$
By Proposition~\ref{P:Jzetaellismackey}, we deduce that if $\sigma$ has central character $\omega$ then
\begin{align*}
\sigma(\ell) &\cong \Res_{\K'}J(\omega,\ell)\\
&\cong \Res_{\K'}\Ind_{\intgroup}^{\K} \hat{\eta}_{\omega,\ell}\\
&\cong \bigoplus_{\gamma\in \K'\backslash \K/\intgroup} \Ind^{\K'}_{\K'\cap \prescript{\gamma}{}{(\intgroup)}}\prescript{\gamma}{}{\hat{\eta}_{\omega,\ell}}.
\end{align*}
Since the determinant maps $\K/\intgroup$ surjectively onto $\RR^\times/(1+\PP^m)(\RR^\times)^2$,  a set of representatives for the double coset space $\K'\backslash \K/\intgroup$ is $S=\{\diag(1,\alpha)\mid \alpha \in \cS_m\}$.  Let $\gamma\in S$; then $\gamma$ normalizes $\intgroup$ so $\K'\cap \prescript{\gamma}{}{(\intgroup)}=\intgroup'$. Using \eqref{E:hatetazetauell}, and noting that $\omega(g_{11})=1$ for all  $g\in \intgroup'$, we compute
$$
\prescript{\gamma}{}{\hat{\eta}_{\omega,\ell}}(g) = \hat{\eta}_{\omega,\ell}(\gamma^{-1}g\gamma) = 
\omega(g_{11})\psi(\varpi^{-\ell}g_{11}^{-1}\alpha g_{12}) = \hat{\eta}'_{\triv, (\alpha,\ell)}(g).
$$
Therefore $\sigma(\ell) = \bigoplus_{\alpha \in \cS_m} I(\triv,\alpha,\ell)$  
and thus this is a decomposition into pairwise nonisomorphic irreducible representations of $\K'$, as required.
\end{proof}

We note that the number of components $\cS_{\lceil \ell/2\rceil}$ is precisely the intertwining number found in Corollary~\ref{C:Sigmaell}, as expected.

We may now deduce our principal theorem, which is a description of the full branching rules to $\K'$ of any  irreducible depth-zero supercuspidal representations of $G'$.

\begin{thm}\label{T:brulesSL2}
    Let $\pi$ be an irreducible depth-zero supercuspidal representation of $G'$.
    If $\pi^{\K'_+}=\{0\}$ then
    $$
    \Res_{\K'}\pi \cong \bigoplus_{\ell\in 1+2\mathbb{Z}_{\geq 0}} \bigoplus_{u\in \cS_{(\ell+1)/2}} I(\triv,u,\ell)
    $$
    are its branching rules to $\K'$.  Otherwise, $\pi=\Ind_{\K'}^{G'}\sigma$ for some cuspidal representation $\sigma$ of $\Sl(2,\mathfrak{f})$ and its branching rules are instead
    $$
    \Res_{\K'}\pi \cong \sigma \oplus \bigoplus_{\ell \in 2\mathbb{Z}_{\geq 1}}\bigoplus_{u\in \cS_{\ell/2}} I(\triv,u,\ell).
    $$
\end{thm}

\begin{proof}
    The decomposition of $\pi$ into components of the form $\sigma(\ell)$ according to parity was given in Corollary~\ref{Cor:MackeycomponentsdepthzeroSL2}. In particular, the supercuspidal representation denoted $\pi_0(\sigma)$ has fixed points under $\K'_+$ whereas $\pi_1(\sigma)$ does not.  Noting that $\sigma(0)$ is simply the inflation of $\sigma$, the rest follows from Corollary~\ref{C:Izetauell}.
\end{proof}

In summary, we have obtained the branching rules for the depth-zero supercuspidal representations of $\Sl(2,F)$ by restricting the same irreducible representation of $\K=\Gl(2,\RR)$ twice:  once,  viewing it as coming from the restriction to $\K$ of a supercuspidal representation of $\Gl(2,F)$; and the other,  using the geometry of nilpotent orbits. We illustrate this with the following extension to our earlier diagram.
\[\begin{tikzcd}
		\pi && {\displaystyle\bigoplus_{\ell\geq 0}\Ind_{B_\ell}^{\K}\sigmaell} &{\sigma \oplus \displaystyle\bigoplus_{\ell\geq 1}J(\omega,\ell)}\\
	{\displaystyle\pi_0(\sigma)\oplus \pi_1(\sigma)} && {\displaystyle \sigma \oplus \bigoplus_{\ell\geq 1}\sigma(\ell)}&{\displaystyle \sigma \oplus \bigoplus_{\ell\geq1}\bigoplus_{u\in \cS_{\lceil \ell/2\rceil}}I(\triv,u,\ell)} 
	\arrow["{\Res_{\K'}}", from=2-1, to=2-3]
	\arrow["{\Res_{\K}}", from=1-1, to=1-3]
	\arrow["{\Res_{G'}}"', from=1-1, to=2-1]
	\arrow["\cong", from=1-3, to=1-4]
	\arrow["{\Res_{\K'}}", from=1-3, to=2-3]
	\arrow["{\Res_{\K'}}", from=1-4, to=2-4]
	\arrow["\cong", from=2-3, to=2-4]
\end{tikzcd}\]
When $p$ is odd,  $\cS_m=\cS=\{1,\ep\}$ for all $m\geq 1$, and we have the same diagram for all Deligne--Lusztig supercuspidal representations.

\section{Some applications} \label{S:8}

\subsection{The growth of \texorpdfstring{$\dim(\pi^{\K'_{2n}})$}{dimpiKn}} \label{SS:8.1}

Since a depth-zero supercuspidal representation has Gelfand Kirillov dimension equal to $1$, it is known that for $n\in \mathbb{Z}$, the value $\dim(V^{K_{2n}})$ asymptotically grows like a polynomial in $q$ of degree $2n+c$ for some constant $c$, up to lower order terms   \cite{BarbaschMoy1997}.
In fact, this polynomial was computed for all irreducible representations of $\Sl(2,F)$, for any $F$, by Henniart and Vign\'eras in \cite[Theorem 7.9]{HenniartVigneras2025}.   
We can recover their formula using Theorem~\ref{Cor:MackeycomponentsdepthzeroSL2}, as follows. Since every irreducible component of $\sigma(\ell)$ has depth $\ell$, we have $\sigma(\ell)\subset \pi^{\K'_{2n}}$ if and only if $\ell<2n$.  Thus for every $n>0$, $\pi_0(\sigma)^{V_{\K'_{2n}}}=\sigma \oplus \bigoplus_{i=1}^{n-1}\sigma(2i)$, whose dimension is $q-1 + (q^2-1)\sum_{i=1}^{n-1}q^{2i-1}=q^{2n-1}-1$, whereas $\pi_1(\sigma)^{\K'_{2n}}=\bigoplus_{i=0}^{n-1}\sigma(2i+1)$, whose dimension is $q^{2n}-1$.

A deeper feature is the growth in the dimensions of the irreducible components.  

\begin{lem}\label{degInducedofNilp}
    For any $\ell\geq 1$, $u\in \cS_m$ and character $\zeta$ of $\RR^\times$ of depth less than $m$, the degree of the representation $I(\zeta,u,\ell)$ is 
    $$\deg I(\zeta, u,\ell)=\begin{cases}
        q^2 - 1 & \text{if $\ell = 1;$}\\
        (q^2 -1)\cdot q^{\ell-1-\lfloor (\ell+1)/4\rfloor} & \text{if $1<\ell \leq 4e$};\\
        \frac12(q^2-1)\cdot q^{\ell-1-e} & \text{if  $\ell \geq 4e+1$}.
    \end{cases}
    $$
\end{lem}
\begin{proof}
We have $\deg I(\zeta,u,\ell)=\deg I(\triv,u',\ell)$ for all characters $\zeta$ and choices $u,u'\in \cS_m$.  By Corollary~\ref{C:Izetauell}, $\deg I(\triv,u,\ell)=\deg(\sigma(\ell))/\Sigma(\ell)$, where $\Sigma(\ell)=|\cS_m|$ was computed in Corollary~\ref{C:Sigmaell}.  A quick calculation yields this form.
 \end{proof}

Observe that for the same depth $\ell$, the dimensions of these representations  for $2$-adic fields $F$ eventually (that is, for $\ell\geq 4e+1$) grow as $\frac12(q^2-1)q^{\ell-1-e}\approx q^{\ell-e+1}$, which is much larger than the dimensions of the corresponding representations for fields $\Fqt$, which is only $(q^2-1)q^{\ell-1-\lfloor (\ell+1)/4\rfloor} \approx q^{\frac34\ell+1}$.   This reflects that for $\ell \geq 4e+1$, the character $\eta_{(u,\ell)}$ extends to a much larger subgroup when $\car(F)=0$ than it can when $\car(F)=2$.

\begin{prop}\label{P:dn}
    Let $n\in \mathbb{Z}_{>0}$ and let $d(n)$ denote the dimension of the largest irreducible component of $\pi^{\K'_n}$, where $\pi$ is an irreducible depth-zero supercuspidal representation of $\Sl(2, F)$. 
    Then we have 
    \begin{equation*}
        \frac{d(n+4)}{d(n)} = \begin{cases}
            q^3 & \text{ if $n \leq 4e - 3$}\\
            \frac{1}{2}q^3 & \text{ if $n \in \{4e - 2, 4e-1\}$}\\
            \frac{1}{2}q^4 & \text{ if $n \in \{4e, 4e+1\}$}\\
            q^4 & \text{ if $n \geq 4e+2$.}
        \end{cases}
    \end{equation*}
    Moreover, this rate of growth satisfies $\{d(n+2)/d(n), d(n)/d(n-2)\} = \{q,q^2\}$ for $3\leq n\leq 4e-1$, whereas for $n \geq 4e + 2$, we have instead that $d(n+2)/d(n) = q^2$.
\end{prop}
\begin{proof}
    Observe that by Theorem~\ref{T:brulesSL2}, we have $d(n) = \dim(I(\triv, 1, n-1))$. By Lemma~\ref{degInducedofNilp} we have that the degrees of $I(\triv, 1, n-1)$ and $I(\triv, 1, n+3)$ are given by the same formula when either $n+3 \leq 4e$ or $n-1 \geq 4e+1$; a quick application of Lemma~\ref{degInducedofNilp} achieves the result. 
    For the remaining intermediate cases, we have $d(n+4)/d(n)=\frac12 q^{4-e + \lfloor n/4 \rfloor}$.  If $n \in \{4e - 2, 4e- 1\}$, then $\lfloor n/4 \rfloor=e-1$ so this value is $\frac12q^3$ whereas if  $n \in \{4e, 4e+1\}$, then $\lfloor n/4 \rfloor=e$, yielding $\frac12q^4$. 
\end{proof}

In contrast, when $p$ is odd the growth rate of irreducible subrepresentations is $d(n+2)/d(n)=q^2$ for all $n\geq 1$ \cite[\S4]{Nev}.

\subsection{A representation-theoretic local character expansion} \label{SS:LCE}

The Harish-Chandra--Howe local character expansion exists when $\car(F)=0$ \cite{HarishChandra1999} or when $p$ is (very) large \cite{CluckersGordonHalupczok2014}.
It asserts that in a neighbourhood of the identity where the exponential map (or substitute) converges, the trace character of an admissible representation $\pi$ can be written as a linear combination of Fourier transforms of the (finitely many) nilpotent orbital integrals.  The maximal nilpotent orbits to occur with nonzero coefficients are called the wavefront set of $\pi$.   The domain of validity of this expansion is also known when the residual characteristic $p$ is sufficiently large \cite{DeBack1}, when it is $\cup_{x\in \buil(G)}G_{x,r+}$ where $r$ is the depth of $\pi$. 

Recent work in \cite{Nevins2024} proposes a representation-theoretic version of the local character expansion: for $\Sl(2,F)$ with $p\neq 2$, \cite[Theorem 7.4]{Nevins2024} explicitly expresses the restriction of any irreducible representation $\pi$ to a sufficiently small open subgroup as a linear combination of representations of that subgroup associated to nilpotent orbits in the Lie algebra.  In that case, the radius of convergence $G_{x,r+}$ was the same as that of the local character expansion, and the wavefront sets coincide.

In the next three subsections, we extend this result for depth-zero supercuspidal representations when $p=2$.  

\subsubsection{The LCE for \texorpdfstring{$\Sl(2,F)$}{SL2F} when \texorpdfstring{$F$}{F} has characteristic zero} \label{SS:LCE1}

 We begin by proving that the analogous expansion holds for depth-zero supercuspidal representations over $2$-adic fields with radius of convergence $G'_{x_0,4e+}$.  
 
\begin{dfn}\label{D:tauozeta}
    For each $u \in \cS \cup \varpi \cS$ and character $\zeta$ of $Z'$, define the representation of $\K'$ attached to a nilpotent orbit $\cO_u$ and central character $\zeta$ to be 
    \begin{equation*}
        \tau(\cO_u, \zeta) = \bigoplus_{\ell \in  -\val(u)+2\Z_{\geq 1} } I(\zeta, u, \ell).
    \end{equation*}
    When $\zeta=\triv$, or when the choice of character is irrelevant, we omit $\zeta$ from the notation, writing $\tau(\cO_u)$ instead.
\end{dfn}

Note that these are infinite-dimensional representations of $\K'$ that are constructed from the $\K'$-orbits of negative depth appearing in the $G'$-orbit $\mathcal{O}$.  They are not disjoint:  when $u\equiv  u'\in \cS_{\lceil \ell/2\rceil}$ --- equivalently, if $X_{u\varpi^{-\ell}}$ and $X_{u'\varpi^{-\ell}}$ represent $\K'$-conjugate degenerate $(-\ell,-\ell/2)$ cosets at $x_0$ --- 
we have $I(\zeta,u,\ell)\cong I(\zeta,u',\ell)$ by Theorem~\ref{T:JzetaellIzetauell}.

\begin{lemma}
    The representation $\tau(\mathcal{O},\zeta)$ is independent of the choice of representatives for $\mathcal{O}$ and of its $\K'$ orbits of negative depth, up to equivalence. 
\end{lemma}

\begin{proof}
Let $\mathcal{O}$ be a nilpotent orbit.  By Lemma~\ref{vertexnilpotent}, 
its elements have even depth at one of the two 
conjugacy 
classes of vertices, and odd depth at the other.    
For each depth $-\ell$ of the correct parity, choose an element $X \in \cO$ of depth $-\ell$;  then by \eqref{E:K'orbits}, it is $\K'$-conjugate to 
$X_{u\varpi^{-\ell}}$ for some $u\in \cS$.  Since $X_{u\varpi^i}$ is $G'$-conjugate to $X_{u'\varpi^j}$ if and only if $|i-j|\in 2\mathbb{Z}$ and  $u\equiv u'\in \cS$, the datum $(u,\ell)$ is completely determined by $\mathcal{O}$.  
\end{proof}

Under the hypothesis $p>3e+1$ (and specifically $p\neq 2$), Barbasch and Moy \cite[Theorem 4.5]{BarbaschMoy1997} compute the wavefront set of a depth-zero representation in terms of the Gelfand--Graev representations in which the components of $\pi^{G'_{x,0+}}$ appear.  For $p=2$ (and $\car(F)=0$) we do not know of a reference for a comparable computation.  We propose the following, which coincides with the wavefront sets of Deligne--Lusztig depth-zero supercuspidal representations for $p$ odd, and is well-defined in any characteristic.

\begin{dfn}\label{D:WFpi}
    Let $\sigma$ be a cuspidal representation of $\Sl(2,\mathfrak{f})$.  Then for $i\in \{0,1\}$, let $\WF(\pi_i(\sigma))$ denote the set of nilpotent orbits attached to $x_i$, or equivalently, the set of nilpotent orbits with parity $i+2\mathbb{Z}$ at $x_0$.
\end{dfn}

This definition is justified by the following theorem. 

\begin{thm}\label{T:LCE}
    Suppose $\car(F)=0$.
    Let $\pi$ be an irreducible supercuspidal representation of $G'$ of depth zero.  Then there exists an integer $n_\pi$ such that in the Grothendieck group of representations, we have
    $$
    \Res_{\K'_{4e+1}}\pi \cong n_\pi \triv + \sum_{\mathcal{O}\in \WF(\pi)} \Res_{\K'_{4e+1}}\tau(\mathcal{O}).
    $$
    Moreover, $4e+1 \geq 5$ is the least depth for which this isomorphism holds.
\end{thm}

\begin{proof}
An irreducible supercuspidal representation has the form $\pi = \pi_i$ with $i\in \{0,1\}$, where $\pi_i$ is compactly induced from $G'_{x_i}$.  Thus by Lemma~\ref{vertexnilpotent}, $\WF(\pi_i)$ is the set of nilpotent orbits with parity $i+2\Z$ at $x_0$.  By Theorem~\ref{T:brulesSL2}, $\Res_{\K'}\pi_i$ has an expansion in terms of representations $I(\triv,u,\ell)$ where the parity of $\ell$ agrees with $i$.  Consequently, the corresponding nilpotent orbits $\cO_{u\varpi^{-\ell}}$ lie in $\WF(\pi_i)$.  

Recall that $\cS_m=\cS$ if and only if $m\geq 2e+1$.  For all $\ell\geq 4e+1$, we have $m=\lceil \ell/2\rceil \geq 2e+1$, so by Theorem~\ref{T:JzetaellIzetauell} the  representations
$I(\triv,u,\ell)$, as $u$ ranges over $\cS$, are distinct.
Thus the set of components of each depth $\ell\geq 4e+1$ are precisely the $|\cS|=|\WF(\pi_i)|$ distinct components of depth $\ell$ in the sum $\bigoplus_{\cO\in \WF(\pi_i)}\tau(\cO,\triv)$.  

When $\ell\leq 4e$,  however, there will be elements $u\neq u'\in \cS$ such that $u\equiv u' \in \cS_{\lceil \ell/2\rceil}$.  In this case, the representation $I(\zeta,u,\ell)$ occurs only once in $\Res_{\K'}\pi$ but at least  twice in $\bigoplus_{\mathcal{O}\in \WF(\pi)} \tau(\mathcal{O},\zeta)$.  On the other hand, since $I(\zeta,u,\ell)$ has depth $\ell\leq 4e$, its restriction to $\K'_{4e+1}$ is trivial.  It follows that the restrictions to $\K'_{\ell}$ of the two sides do not agree, for any $\ell\leq 4e$.   

Thus upon restriction to $\K'_{4e+1}$, we obtain the desired equality in the Grothendieck group by setting
\begin{equation}\label{E:formulanpi}
n_{\pi_i} = \dim(\pi_i^{\K'_{4e+1}})- \sum_{j=1}^{2e}\sum_{u\in \cS}\dim I(\triv,u,2j-i) <0.
\end{equation}
At depth $4e+1$, the two characters of $Z'$ coincide, so we write simply $\tau(\cO)$ for $\tau(\cO,\triv)$. 
\end{proof}

The integer $n_\pi$ is readily computable; see for example \eqref{E:q2calcnpi} below for the case $F_0=\mathbb{Q}_2$. 

Theorem~\ref{T:LCE} is a representation-theoretic analogue of the local character expansion, in the sense that it is an equality of representations in a neighbourhood of the identity, whose trace recovers the local character expansion where this exists.

\begin{rem}
    The analogous statement is proven to hold for all local nonarchimedean fields $F$ with odd residual characteristic in \cite[Theorem 1.1]{Nevins2024}.  There, since $e=1$, we have $\K_{4e+1}'=\K_1'=\K'_+$, which coincides with the domain of convergence of the local character expansion.    
\end{rem}

\subsubsection{The case of \texorpdfstring{$\Gl(2,F)$}{GL2F}, for any \texorpdfstring{$F$}{F}} \label{SS:GL2LCE}
The approach of the preceding subsection also applies to the depth-zero supercuspidal representations of $G=\Gl(2,F)$, for both $\car(F)=0$ and $\car(F)=2$.  In this case, there is only one nonzero nilpotent orbit $\cO$, giving rise to one representation
$$
\tau_{\Gl}(\cO,\omega)=\bigoplus_{\ell\in \mathbb{Z}_{>0}}J(\omega,\ell)
$$
for each character $\omega$ of $Z_0$ of depth zero. Then Theorem~\ref{Thm:Hansendepthzero} implies that if $\pi$ is a depth-zero supercuspidal representation of $\Gl(2,F)$, then
$$
\Res_{\K_+}\pi \cong (1-q)\triv \oplus \Res_{\K_+} \tau_{\Gl}(\cO,\triv).
$$
That is, the representation-theoretic version of the local character expansion holds at depth zero
(as does the local character expansion itself, with a mock exponential map in place of $\exp$ \cite{Lemaire1996}).  Alternatively, if the central character of $\pi$ is $\omega$, then we can express the branching rules in this case as
$$
\Res_{\K}\pi \cong \pi^{\K_+} \oplus \tau_{\Gl}(\cO,\omega).
$$
It is this perspective we apply to $\Sl(2,\Fqt)$ in the next section.

\subsubsection{A local expansion for \texorpdfstring{$\Sl(2,F)$}{SL2F} for all \texorpdfstring{$F$}{F}} \label{SS:LCE2}
We now turn to an alternative formulation of the local behaviour of the representations $\pi$.  This one is valid also in the case of $\car(F)=2$, where the local character expansion does not exist and it is even unknown if the nilpotent orbital integrals converge. 

We begin by appropriately collecting the ``close cousins" described in Lemma~\ref{countNilpOrbits}.  That is, for each fixed pair $(u,\ell)$, let $\WF(u,\ell)$ be the set of nilpotent $G'$-orbits meeting $X_{u\varpi^{-\ell}}+\g'_{x_0,-\ell/2}$. Then the following set $T_{u,\ell}$ indexes the \emph{distinct} representations $I(\zeta,u',\ell')$ of depth $\ell'\geq \ell$ occuring in $\bigoplus_{\cO\in \WF(u,\ell)}\tau(\cO,\zeta)$.

\begin{lem}\label{L:Tuell}
    Let $\ell\geq 1$ and $u\in \cS_{\lceil \ell/2 \rceil}$.  Then the set
    $$T_{u,\ell}=\{(u',\ell')\mid \ell'-\ell\in 2\mathbb{Z}_{\geq 0}, u'\in \cS_{\lceil \ell'/2\rceil}, u\equiv u'\mod \PP^{\lceil \ell/2 \rceil}\}$$ indexes all the nilpotent $\K'$-orbits of depths $-\ell'\leq -\ell$, up to equivalence modulo depth $-\ell'/2$, whose corresponding $G'$-orbits meet the $(-\ell,-\ell/2)$ degenerate coset of $X_{u\varpi^{-\ell}}$ at $x_0$.
\end{lem}

\begin{proof}     
Let $(u',\ell')\in T_{u,\ell}$ and set $2n=\ell'-\ell$.  Choose $\alpha \in \RR^\times$ and $\beta\in 1+\PP^{\lceil \ell/2 \rceil}$ such that $u'=u\alpha^2\beta$.   Then with 
    $g=\diag(\alpha\varpi^{n},\alpha^{-1}\varpi^{-n})$ we have 
    $gX_{u'\varpi^{-\ell'}}g^{-1}\in X_{u\varpi^{-\ell}}+\g_{x_0,-\ell/2}$, so the orbit of $X_{u'\varpi^{-\ell'}}$ meets the required degenerate coset.  
    Moreover, choosing $u'$ to range over $\cS_{\lceil \ell'/2\rceil}$ yields that the corresponding $\K'$-orbits are pairwise inequivalent modulo depth $-\ell'/2$. 
    Conversely, by Lemma~\ref{countNilpOrbits}, every $G'$-orbit meeting $X_{u\varpi^{-\ell}}+\g_{x_0,-\ell/2}$ is represented by $X_{u'\varpi^{-\ell}}$ for some $u'\in \cS$ that is equivalent to $u$ modulo $\PP^{\lceil \ell/2\rceil}$, and the $\K'$ orbits of these are exactly those represented by $X_{u'\varpi^{-\ell-2n}}$ for some $n\geq 0$.
\end{proof}

Note that for each fixed $\ell$, the intersection $T_{u,\ell}\cap T_{u',\ell}$ is nonempty if and only if $u\equiv u'\in \cS_{\lceil \ell/2\rceil}.$   
Moreover, for $\ell'>\ell$, $T_{u',\ell'}\subseteq T_{u,\ell}$ if and only if $(u',\ell')\in T_{u,\ell}$ and otherwise they are disjoint. 

\begin{dfn}\label{D:tauzetauell}
Given a
 triple $(\zeta, u,\ell)$, with $\ell\geq 1$, $m=\lceil \ell/2 \rceil$, $u\in \cS_m$ and character $\zeta$ of depth less than $m$, define the representation of $\K'$ associated to $\zeta$ and the degenerate $(-\ell,-\ell/2)$ coset of $X_{u\varpi^{-\ell}}$ at $x_0$ by
$$
\tau_{\zeta, u, \ell} = \bigoplus_{(u',\ell')\in T_{u,\ell}} I(\zeta, u',\ell').
$$
\end{dfn}

By construction, the (infinitely many) summands of $\tau_{\zeta,u,\ell}$ are all pairwise nonisomorphic; their depths are all greater than or equal to $\ell$ and of the same parity as $\ell$.  By the discussion above, for each fixed $\ell$, $\tau_{\zeta, u, \ell}$ and $\tau_{\zeta, u', \ell}$ are disjoint whenever $u$ and $u'$ are distinct in $\cS_m$. 

We now arrive at an analogue for $p=2$ of \cite[Theorem 7.4]{Nevins2024}, one that holds when $\car(F)=0$ or $\car(F)=2$.  

\begin{thm}\label{T:LCE2}
Let $\pi = \pi_i(\sigma)$ be a depth-zero supercuspidal representation of $\Sl(2,F)$ where $\car(F)\in \{0,2\}$ and $p=2$.  Then for any $\ell>0$ such that $\ell \in i+2\mathbb{Z}$ we have
$$
\Res_{\K'}\pi \cong \pi^{\K'_{\ell}} \oplus \bigoplus_{u\in \cS_{\lceil \ell/2 \rceil}} \tau_{\triv,u,\ell}.
$$
The number of distinct summands $\tau_{\triv, u, \ell}$ is $|\mathcal{S}_{\lceil \ell/2\rceil}|$; in particular, it is $2q^e$ if $\ell\geq 4e+1$ but grows to infinity with $\ell$ when $\car(F)=2$ and $e=\infty$.
\end{thm}

\begin{proof}
    By Corollary~\ref{C:Izetauell}, we find that for all $\ell'\geq \ell$ with $\ell'-\ell\in 2\Z$,  $\sigma(\ell')$ intertwines with $\tau_{\triv, u',\ell'}$ for each $u\in \mathcal{S}_{\lceil \ell'/2\rceil}$ and by Corollary~\ref{C:Sigmaell} all irreducible components of $\sigma(\ell')$ occur as a depth $\ell'$ component of some $\tau_{\triv, u, \ell}$.  The result follows since the $\tau_{\triv, u, \ell}$ are disjoint.
\end{proof}

The expansion in Theorem~\ref{T:LCE2} captures the local triviality of representations near the identity, even when $\car(F)=2$.  In fact, for any $\ell>0$ we recover a version of Theorem~\ref{T:LCE} that is well-defined, even in characteristic two.   Namely, since $\pi^{\K'_\ell}$ restricts trivially to $\K'_{\ell}$, we recover
\begin{equation}\label{LCE char 2}
    \Res_{\K_\ell'}\pi \cong 
    \dim(\pi^{\K_{\ell}'}) \cdot \triv \oplus \bigoplus_{u\in S_{\lceil \ell/2\rceil}}\Res_{\K'_{\ell}}\tau_{\triv, u, \ell}. 
\end{equation}
 This gives a family of distinct decompositions indexed by the depth $\ell$ of the restriction, in which the components correspond to (finitely many!) equivalence classes of orbits of $\WF(\pi)$, where this equivalence is characterized by the $\K'$-orbits of its $(-\ell,-\ell/2)$ degenerate cosets at $x_0$.

\subsection{Explicit results for \texorpdfstring{$\mathbb{Q}_2$}{Q2}}\label{SS:8.4Q2}

Let us apply our results to the special but interesting case of $\Sl(2,\mathbb{Q}_2)$.  In this case, $e=1$, $\varpi=2$, $q=2$ and $|\cS|=4$.
Thus by Theorem~\ref{dimHomeSigmaEll}, the double cosets supporting intertwining operators take the form
\begin{equation}\label{E:dcstQ2}
\dcst_{\ell, sup}
=\{I\} 
\cup \{g(\ell-1,1) \mid \text{if $\ell\geq 3$}\} \cup \{g(\ell-2,1), g(\ell-2, 1+\varpi) \mid \text{if $\ell \geq 5$}\}
\end{equation}
and each intertwining number is $1$.  By Corollary~\ref{C:Sigmaell}, and the fact that the decomposition is multiplicity-free, we have that each $\sigma(\ell)$ correspondingly decomposes as a direct sum of $\Sigma(\ell)\in \{1,2,4\}$ irreducible subrepresentations, whose degrees are given in Lemma~\ref{degInducedofNilp}.  We summarize this in Table~\ref{Table:Q2reps}. 

\begin{table}[ht]
    \begin{tabular}{|ccc|ccc|}
\hline    \multicolumn{3}{|c}{$\pi_0 = \cInd_{G'_{x_0}}^{G'}\sigma$} & \multicolumn{3}{|c|}{$\pi_1 = \cInd_{G'_{x_1}}^{G'}\prescript{g_1}{}{\sigma}$}\\
    depth & \# components & degree & depth & \# components & degree\\ \hline
    0 & 1 & 1 &&&\\
    2 & 1 & 6 & 1 & 1 & 3\\
    4 & 2 & 12 & 3 & 2 & 6\\
    $6=4e+2$ & 4 & 24 & $5=4e+1$ & 4 & 12\\
    $2k, k\geq 4$ & 4 & $3\cdot 2^{2k-3}$ & $2k+1, k\geq 3$ & 4 & $3\cdot 2^{2k-2}$\\ \hline
    \end{tabular}
    \caption{\label{Table:Q2reps}The number and degree of irreducible representations of $\Res_{\K'}\pi$, for $\pi$ as supercuspidal representation of $\Sl(2,\mathbb{Q}_2)$. Note that the depth-zero component is $\sigma$, a type for $\pi_0$.}
\end{table}

\begin{rem}
In contrast, when $p$ is odd, every positive-depth Mackey component of $\Res_{\K'}\pi$, for an irreducible Deligne--Lusztig supercuspidal representation of $\Sl(2,F)$, decomposes as a direct sum of exactly two irreducible subrepresentations of degree $\frac12(q^2-1)q^{\ell-1}$ \cite[Theorem 5.3]{Nev}.
\end{rem}

We may also conclude that when $\ell \geq 5=4e+1$, we have
$\End_{\K'}(\sigma(\ell)) \cong \mathbb{C}[\mathbb{Z}/4\mathbb{Z}]$ for all $\ell\geq 5=4e+1$, as follows.  By Proposition~\ref{actionOnEndSpaceQ2}, the operator $\mathcal{F}_{g(\ell-2,\alpha)}$ with $\alpha \in \{1,1+\varpi\}$ acts as $\beta \mapsto \beta + \alpha\varpi^{\ell-2}$ modulo $\PP^\ell$.  Since $\mathbb{Z}\alpha\equiv \{0,\alpha, 2\alpha, 3\alpha\} \mod \PP^2$, this operator has order $4$.  The remaining operators  have order $1$ or $2$.

In this case, Theorem~\ref{T:LCE} expresses $\pi_i$ as a linear combination of four of the eight representations associated to the nilpotent orbits.  For example, suppose $\WF(\pi)=\{\cO_u\mid u\in \cS\}$.  Then for each orbit we have
$$
\tau(\cO_u,\triv) = I(\triv,u,2) \oplus I(\triv,u,4) \oplus I(\triv,u,6) \oplus \cdots
$$
where most of these components are distinct, except that  for all $u$ we have $I(\triv,u,2)\cong I(\triv,1,2)$  and $I(\triv,u,4) \cong I(\triv, u+\varpi^2, 4)$.
Then since by Table~\ref{Table:Q2reps},  $\dim(\pi_0^{\K'_5})=31$ and $\dim(I(\triv,u,2))+\dim(I(\triv,u,4))=6+12=18$, we have by \eqref{E:formulanpi} that
\begin{equation}\label{E:q2calcnpi}
\Res_{\K_{5}'}\pi_0 = -41\cdot  \triv + \sum_{\cO\in \WF(\pi_0)}\Res_{\K'_5}\tau(\cO).
\end{equation}
Similarly, $\Res_{\K_5'}(\pi_1)= -21\cdot \triv + 
\sum_{u\in \WF(\pi_1)} \Res_{\K'_5}\tau(\cO)$.

On the other hand, Theorem~\ref{T:LCE2} allows us to write variously, for example,
$$
\Res_{\K'}\pi_0 \cong \pi^{\K'_2} \oplus \tau_{\triv,1,2} \cong \pi^{\K'_4} \oplus \tau_{\triv,1,4}\oplus \tau_{\triv,1+\varpi,4}. 
$$
When instead  $F=\Fpt$, then this pattern continues indefinitely.

\providecommand{\bysame}{\leavevmode\hbox to3em{\hrulefill}\thinspace}
\providecommand{\MR}{\relax\ifhmode\unskip\space\fi MR }
% \MRhref is called by the amsart/book/proc definition of \MR.
\providecommand{\MRhref}[2]{%
  \href{http://www.ams.org/mathscinet-getitem?mr=#1}{#2}
}
\providecommand{\href}[2]{#2}

\end{document}